\documentclass[10pt, reqno]{amsart}

\usepackage[utf8]{inputenc}
\usepackage{amsmath}
\usepackage{amssymb}
\usepackage{mathrsfs}
\usepackage{graphics}
\usepackage[dvipdf]{graphicx}
\usepackage{geometry}
\usepackage{setspace}
\usepackage{color}
\usepackage{amsthm}
\usepackage{paralist}
\usepackage{dsfont}
\usepackage{verbatim}
\usepackage[numbers]{natbib}

\newcommand{\Rd}{{\mathbb{R}^d}}
\newcommand{\calH}{{\mathcal{H}}}

\newcommand{\eps}{\varepsilon}
\newcommand{\Rdrei}{{\mathbb{R}^3}}
\newcommand{\W}{\mathbf{W}}
\newcommand{\m}{\mathbf{m}}
\newcommand{\N}{{\mathbb{N}}}
\newcommand{\R}{{\mathbb{R}}}
\newcommand{\flow}[2]{\mathsf{S}_{#1}^{#2}}
\newcommand{\eins}{\mathds{1}}

\newcommand{\bphi}{\overline{\phi''}}
\newcommand{\dd}{\,\mathrm{d}}
\newcommand{\dn}{\mathrm{d}}
\newcommand{\dv}{\mathrm{div}}
\newcommand{\dff}{\mathrm{D}}
\newcommand{\grd}{\nabla}
\newcommand{\dist}{\mathrm{\mathbf{dist}}}
\newcommand{\prb}{\mathscr{P}_2}
\newcommand{\prbb}{\mathscr{P}}
\newcommand{\auxil}{\mathfrak{A}}
\newcommand{\metr}{\mathbf{d}}
\newcommand{\theX}{\mathbf{X}}
\newcommand{\lyp}{\mathcal{L}}
\newcommand{\aU}{\mathcal{U}}
\newcommand{\aV}{\mathcal{V}}
\newcommand{\ent}{\mathcal{E}}
\newcommand{\krnl}{\mathbf{G}}
\newcommand{\dsp}{\mathcal{D}}
\newcommand{\yucca}{\mathbf{Y}}
\newcommand{\heat}{\mathbf{H}}
\newcommand{\bv}{\hat v}

\newtheoremstyle{meinstyle_satz}
{}
{}
{\itshape}
{}
{}
{}
{1ex}
{\textbf{\thmname{#1}}\thmnumber{ \textbf{#2}}\textbf{.} (\textit{\thmnote{#3}})\newline}

\newtheoremstyle{meinstyle_rest}
{}
{}
{}
{}
{}
{}
{1ex}
{\textbf{\thmname{#1}}\thmnumber{ \textbf{#2}}\textbf{.} (\textit{\thmnote{#3}})\newline}

\newtheorem{definition}{Definition}[section]
\newtheorem{remark}[definition]{Remark}
\newtheorem{lemma}[definition]{Lemma}
\newtheorem{thm}[definition]{Theorem}
\newtheorem{prop}[definition]{Proposition}

\newtheorem{coro}[definition]{Corollary}

\begin{document}

\begin{abstract}
 We study a system of two coupled nonlinear parabolic equations.
 It constitutes a variant of the Keller-Segel model for chemotaxis, i.e.
 it models the behaviour of a population of bacteria that interact by means of a signalling substance.
 We assume an external confinement for the bacteria and a nonlinear dependency of the chemotactic drift on the signalling substance concentration.

 We perform an analysis of existence and long-time behaviour of solutions based on the underlying gradient flow structure of the system.
 The result is that, for a wide class of initial conditions, weak solutions exist globally in time
 and converge exponentially fast to the unique stationary state under suitable assumptions on the convexity of the confinement 
 and the strength of the coupling.
\end{abstract}

\title[Convergence to equilibrium]{Exponential convergence to equilibrium in a coupled gradient flow system modelling chemotaxis}
\author[Jonathan Zinsl]{Jonathan Zinsl}
\author[Daniel Matthes]{Daniel Matthes}
\address{Zentrum f\"ur Mathematik \\ Technische Universit\"at M\"unchen \\ 85747 Garching, Germany}
\email{zinsl@ma.tum.de}
\email{matthes@ma.tum.de}
\thanks{This research was supported by the German Research Foundation (DFG), Collaborative Research Center SFB-TR 109.}
\date{\today}
\maketitle

\section{Introduction}\label{sec:intro}
\subsection{The equations and their variational structure}
This paper is concerned with existence and long-time behaviour of weak nonnegative solutions to the following initial value problem:
\begin{align}
 \label{eq:pde_u}
 \partial_t u(t,x)&=\mathrm{div}(u(t,x)\dff\left[u(t,x)+W(x)+\eps \phi(v(t,x))\right]),
 \\
 \label{eq:pde_v}
 \partial_t v(t,x)&=\Delta v(t,x)-\kappa v(t,x)-\eps u(t,x)\phi'(v(t,x)),
 \\
 \label{eq:init}
 &u(0,x) = u_0(x) \ge 0,\quad v(0,x) = v_0(x)\ge0,
\end{align}
where the sought functions $u$ and $v$ are defined for $(t,x)\in[0,\infty)\times \Rdrei$.
Below, we comment in detail on the origin of \eqref{eq:pde_u}\&\eqref{eq:pde_v} from mathematical biology.
In brief, $u$ is the spatial density of bacteria that interact with each other by means of a signalling substance
of local concentration $v$.

In \eqref{eq:pde_u}\&\eqref{eq:pde_v}, $\eps$ and $\kappa$ are given positive constants;
we are mainly concerned with the case where the \emph{coupling strength} $\eps$ is sufficiently small. Strict positivity of $\kappa$ is essential for our approach, as explained below.
The \emph{response function} $\phi\in C^2([0,\infty))$ is assumed to be convex and strictly decreasing,
with
\begin{align}
 \label{eq:assumpphi}
 0<-\phi'(w)\le-\phi'(0)<\infty\ \text{and}\ 0\le\phi''(w)\le\bphi<\infty \quad\text{for all $w\ge0$},
\end{align}
for an appropriate constant $\bphi\ge 0$,
the paradigmatic examples being
\begin{align}
 \phi(w)&=-w\qquad\text{(classical Keller-Segel model)},\label{eq:classicalKS}\\
 \phi(w)&=-\log(1+w)\qquad\text{(weak saturation effect)},\\
 \phi(w) &= \frac{1}{1+w}\qquad\text{(strong saturation effect)}.\label{eq:paradigm}
\end{align}
The \textit{external potential} $W\in C^2(\Rdrei)$ is assumed to grow quadratically: It has globally bounded second order partial derivatives and is uniformly convex with a constant $\lambda_0>0$,
that is
\begin{align}
 \label{eq:assumpW}
 \dff^2W(x)\ge\lambda_0\eins \quad \text{for all $x\in\Rdrei$, in the sense of symmetric matrices}.
\end{align}
Without loss of generality, we may assume that $W\ge0$.
\medskip

Equations \eqref{eq:pde_u}\&\eqref{eq:pde_v} possess a variational structure.
Formally, they can be written as a gradient flow of the entropy functional
\begin{align*}
 \calH(u,v)&:=\int_\Rdrei \left(\frac{1}{2}u^2+uW+\frac{1}{2}|\dff v|^2+\frac{\kappa}{2}v^2+\eps u\phi(v)\right)\dd x
\end{align*}
with respect to a metric $\dist$, defined on the space $\theX:=\prb(\Rdrei)\times L^2_+(\Rdrei)$ by
\begin{align}
  \label{eq:dist}
  \dist((u_1,v_1),(u_2,v_2))&:=\sqrt{\W_2^2(u_1,u_2)+\|v_1-v_2\|_{L^2(\Rdrei)}^2}
  \quad \text{for $(u_1,v_1),(u_2,v_2)\in\theX$}.
\end{align}
Here $\W_2$ is the \emph{$L^2$-Wasserstein metric} on the space $\prb(\Rdrei)$ of probability measures on $\Rdrei$ with finite second moment,
see Section \ref{subsec:wasser} for the definition.
This gradient flow structure is at the basis of our proof for global existence of weak solutions to \eqref{eq:pde_u}--\eqref{eq:init},
and it is also the key element for our analysis of long-time behaviour.
We remark that, even with this variational structure at hand, the analysis is far from trivial,
since $\calH$ is \emph{not} convex along geodesics. 
Therefore, the established general theory on $\lambda$-contractive gradient flows in metric spaces \cite{savare2008} is not directly applicable.

\subsection{Statement of the main results}
In the first part of this work, we show that a weak solution to \eqref{eq:pde_u}\&\eqref{eq:pde_v} can be obtained 
by means of the time-discrete implicit Euler approximation (also known as \emph{minimizing movement} or \emph{JKO scheme}). More precisely,
for each sufficiently small time step $\tau>0$,
let $(u_\tau^0,v_\tau^0):=(u_0,v_0)$, and then define inductively for each $n\in\N$:
\begin{align}
  \label{eq:jko}
  (u_\tau^n,v_\tau^n) \in \operatorname*{argmin}_{(u,v)\in\prb(\Rdrei)\times L^2(\Rdrei)} \Big(\frac1{2\tau}\dist\big((u,v),(u_\tau^{n-1},v_\tau^{n-1})\big)^2 + \calH(u,v)\Big).
\end{align}
We will prove in Section \ref{subsec:discr_constr} that this construction is well-defined, i.e., that a minimizer exists for every $n\in\N$.
Further, introduce the piecewise constant interpolation $(u_\tau,v_\tau):\R_+\to\prb(\Rdrei)\times L^2(\Rdrei)$ by
\begin{align}
  \label{eq:interpol}
  u_\tau(t) = u_\tau^n,\quad v_\tau(t) = v_\tau^n \quad \text{for all $t\in((n-1)\tau,n\tau]$}.
\end{align}
Our existence result -- which does not require a small coupling strength $\eps$ -- reads as follows:
\begin{thm}[Existence of weak solutions to \eqref{eq:pde_u}\&\eqref{eq:pde_v}]
  \label{thm:existence}
  Let $\kappa>0$ and $\eps>0$ be given,
  assume that the response function $\phi$ satisfies \eqref{eq:assumpphi},
  and that the convex confinement potential $W$ grows quadratically.
  
  Let further initial conditions $u_0\in\prb(\Rdrei)\cap L^2(\Rdrei)$ and $v_0\in W^{1,2}(\Rdrei)$ be given, with $v_0\ge0$,
  and define for each $\tau>0$ a function $(u_\tau,v_\tau)$ by means of the scheme \eqref{eq:jko} and \eqref{eq:interpol}.
  Then, there is a sequence $(\tau_k)_{k\in\N}$ with $\tau_k\downarrow0$ such that $(u_{\tau_k},v_{\tau_k})$ converges to a weak solution $(u,v):[0,\infty)\times\Rdrei\to [0,\infty]^2$ of \eqref{eq:pde_u}--\eqref{eq:init}, in the following sense:
  \begin{align*}
	u_{\tau_k}&\to u \text{ narrowly in $\mathscr{P}(\Rdrei)$ pointwise with respect to $t\in[0,T]$},\\
	v_{\tau_k}&\to v \text{ in $L^2(\Rdrei)$ uniformly with respect to $t\in[0,T]$},\\
    u&\in C^{1/2}([0,T],\prb(\Rdrei))\cap L^\infty([0,T],L^2(\Rdrei))\cap L^2([0,T],W^{1,2}(\Rdrei)),\\
    v&\in C^{1/2}([0,T],L^2(\Rdrei))\cap L^\infty([0,T],W^{1,2}(\Rdrei))\cap L^2([0,T],W^{2,2}(\Rdrei))\cap W^{1,2}([0,T],L^2(\Rdrei))
  \end{align*}
  for all $T>0$, 
  and $(u,v)$ satisfies:
  \begin{align}
   \label{eq:pde_weak1}
  \partial_t u &= \dv\big( u\dff\big[u+W+\eps\phi(v)\big]\big)\quad\text{ in the sense of distributions},
   \\
   \label{eq:pde_weak2}
  \partial_t v &= \Delta v-\kappa v-\eps u\phi'(v) \quad\text{ a.e. in $(0,+\infty)\times\Rdrei$},
\\
u(0)&=u_0\text{ and }v(0)=v_0.
 \end{align}
\end{thm}
The convergence of $(u_{\tau_k},v_{\tau_k})$ is actually much stronger; see Proposition \ref{prop:conv_disc} for details.
%
The key \textit{a priori} estimate yielding sufficient compactness of $(\bar u_\tau,\bar v_\tau)$ follows from a dissipation estimate, 
which formally amounts to
\begin{align*}
  -\frac{\dn}{\dn t}\int_\Rdrei \Big(u\log u + \frac12|\dff v|^2 + \frac\kappa2v^2\Big)\dd x
  &\ge \frac12 \int_{\Rdrei} \big(|\dff u|^2 + (\Delta v-\kappa v)^2\big)\dd x \\
  & \qquad - C\big(\|u\|_{L^2(\Rdrei)}^2 + \|v\|_{W^{1,2}(\Rdrei)}^2 + \|\Delta W\|_{L^\infty(\Rdrei)}\big).
\end{align*}
Related existence results have been proved recently for similar systems of equations, using essentially the same technique,
in \cite{laurencot2011,blanchet2012,zinsl2012}.
Therefore, we keep the technical details to a minimum.

Our main result is the following on the long-time behaviour of solutions.
\begin{thm}[Exponential convergence to equilibrium]\label{thm:convergence}
  Let $\kappa$, $\phi$ and $W$ be as in Theorem \ref{thm:existence} above.
  Then there are constants $\bar\eps>0$, $L>0$ and $C>0$ such that
  for every $\eps\in(0,\bar\eps)$ and with $\Lambda_\eps:=\min(\kappa,\lambda_0)-L\eps$,
  the following is true:

  Let initial conditions $u_0\in\prb(\Rdrei)\cap L^2(\Rdrei)$ and $v_0\in W^{1,2}(\Rdrei)$ be given, with $v_0\ge0$,
  and assume in addition that $v_0\in L^{6/5}(\Rdrei)$.
  Let further $(u,v)$ be a weak solution to \eqref{eq:pde_u}--\eqref{eq:init} obtained as a limit of the scheme \eqref{eq:jko} and \eqref{eq:interpol}.
  Then $(u,v)$ converges to the unique nonnegative stationary solution $(u_\infty,v_\infty)\in(\prb\cap L^2){(\Rdrei)}\times W^{1,2}(\Rdrei)$ of \eqref{eq:pde_u}\&\eqref{eq:pde_v} 
  exponentially fast with rate $\Lambda_\eps$ in the following sense:
  \begin{align}
    \label{eq:theclaim}
    \W_2(u(t,\cdot),u_\infty) &+ \|u(t,\cdot)-u_\infty\|_{L^2(\Rdrei)} + \|v(t,\cdot)-v_\infty\|_{W^{1,2}(\Rdrei)} \nonumber\\
    &\le C\left(1+\|v_0\|_{L^{6/5}(\Rdrei)}\right)\left(\calH(u_0,v_0)-\calH(u_\infty,v_\infty)+1\right)e^{-\Lambda_\eps t} \qquad \text{for all $t\ge0$}.
  \end{align}
\end{thm}

We give a brief and formal indication of the main idea for the proof of Theorem \ref{thm:convergence}.
First, we decompose the entropy in the form
\begin{align}
 \label{eq:decompose}
 \calH(u,v)-\calH(u_\infty,v_\infty)  = \lyp_u(u) + \lyp_v(v) + \eps\lyp_*(u,v),
\end{align}
such that $\lyp_u$ and $\lyp_v$ are $\lambda_\eps$-convex and $\kappa$-convex functionals in $(\prb,\W_2)$ and in $L^2$, respectively,
which are minimized by the stationary solution $(u_\infty,v_\infty)$;
the functional $\lyp_*$ has no useful convexity properties.
On a very formal level --
pretending that $\lyp_u$, $\lyp_v$ and $\lyp_*$ are smooth functionals on Euclidean spaces and denoting their ``gradients'' by $\grd_u$ and $\grd_v$
-- the dissipation of the \emph{principal entropy} $\lyp_u+\lyp_v$ amounts to
\begin{align}
 \label{eq:veryformal}
 \begin{split}
   -\frac{\dn}{\dn t}\big(\lyp_u+\lyp_v\big)
   &= \grd_u\lyp_u\cdot\grd_u\calH + \grd_v\lyp_v\cdot\grd_v\calH \\
   &= \|\grd_u\lyp_u\|^2 + \|\grd_v\lyp_v\|^2 + \eps\grd_u\lyp_u\cdot\grd_u\lyp_* + \eps\grd_v\lyp_v\cdot\grd_v\lyp_* \\
   &\ge (1-\eps)\|\grd_u\lyp_u\|^2 + (1-\eps)\|\grd_v\lyp_v\|^2 - \frac\eps2 \big(\|\grd_u\lyp_*\|^2 + \|\grd_v\lyp_*\|^2\big).
 \end{split}
\end{align}
By convexity of $\lyp_u$ and $\lyp_v$, one has the inequalities
\begin{align*}
 \|\grd_u\lyp_u\|^2 \ge 2\lambda_\eps\lyp_u, \quad
 \|\grd_v\lyp_v\|^2 \ge 2\kappa\lyp_v,
\end{align*}
and so we are \emph{almost} in the situation to apply the Gronwall estimate to \eqref{eq:veryformal}
and conclude convergence to equilibrium with an exponential rate of $\min(\lambda_\eps,\kappa)>0$.
However, it remains to estimate the terms involving the ``gradients'' of $\lyp_*$.
This is relatively straightforward if the entropy $\calH(u,v)$ is sufficiently close to its minimal value $\calH(u_\infty,v_\infty)$,
but is rather difficult for $(u,v)$ far from equilibrium. Moreover, rigorous estimates have to be carried out on the time-discrete level (with subsequent passage to continuous time) since our notion of solution is too weak to carry out the respective estimates in continuous time.\\

In the language of gradient flows, our results can be interpreted as follows.
For $\eps=0$, the functional $\calH$ is $\Lambda_0$-convex along geodesics in $(\theX,\dist)$, with $\Lambda_0=\min(\lambda_0,\kappa)>0$.
Consequently, there is an associated $\Lambda_0$-contractive gradient flow defined on all of $\theX$ which satisfies \eqref{eq:pde_u}\&\eqref{eq:pde_v},
and in particular all solutions converge with the exponential rate $\Lambda_0$ to the unique equilibrium.
For every $\eps>0$, the convexity of $\calH$ is lost; see \cite{zinsl2012} for a discussion of (non-)convexity in a similar situation. 
By Theorem \ref{thm:existence}, equations \eqref{eq:pde_u}\&\eqref{eq:pde_v} still define a continuous flow on the proper domain of $\calH$, which is $\theX\cap(L^2(\Rdrei)\times W^{1,2}(\Rdrei))$.
Further, we show that on the (almost exhaustive) subset of those $(u,v)$ with $v\in L^{6/5}(\Rdrei)$,
this flow still converges to an equilibrium with an exponential rate $\Lambda_\eps\ge\Lambda_0-L\eps>0$
-- see Theorem \ref{thm:convergence}, for all $\eps>0$ sufficiently small.

From this point of view, our result is perturbative: The uncoupled system ($\eps=0$) exhibiting a strictly contractive flow is perturbed in such a way that the perturbed system ($\eps>0$) still yields exponential convergence towards the unique equilibrium -- with a slightly slower convergence rate than in the unperturbed case. For this approach to work, we obviously need to require $\kappa>0$ and $\lambda_0>0$. 
On the other hand, this theorem is stronger than a usual perturbation result:
The crucial point is that we do \emph{not} require the initial condition $(u_0,v_0)$ to be close to equilibrium,
apart from the rather harmless additional hypothesis that $v_0\in L^{6/5}(\Rdrei)$, which could be weakened further with additional technical effort.

Our result on global existence of weak solutions (Theorem \ref{thm:existence}), however, can be generalized to the case of $\kappa=0$ and no convexity assumption on the confinement potential (see e.g. \cite{zinsl2012}). Further generalization of Theorem \ref{thm:existence} to the case of nonlinear, but non-quadratic diffusion can be achieved with similar techniques as in \cite{zinsl2012}. However, in our analysis of the long-time behaviour, the right entropy dissipation estimates are not at hand at the best of our knowledge when dealing with non-quadratic diffusion. To keep technicalities to a minimum, we consider the quadratic case throughout this work.

We expect that similar results can be proved for system \eqref{eq:pde_u}\&\eqref{eq:pde_v} on a bounded domain $\Omega\subset\R^3$, even with vanishing confinement $W\equiv 0$. The role of the confinement will then be played by Poincaré's inequality. Our setup with a convex confinement on $\Rdrei$ fits much more naturally into the the variational framework.

\subsection{Modelling background}
The system of equations \eqref{eq:pde_u}\&\eqref{eq:pde_v} is a variant of the so-called \textit{Keller-Segel model for chemotaxis} 
describing the time-dependent distribution of biological cells or microorganisms in response to gradients of chemical substances (\textit{chemotaxis}). 
The original model -- corresponding to the linear response function from \eqref{eq:classicalKS} -- has been developed by Keller and Segel in \cite{keller1970} as description of slime mold aggregation. 
However, chemotactic processes occur in many (and highly different) biological systems; for the biological details, we refer to the book by Eisenbach \cite{eisenbach2004}. 
For example, many bacteria like \textit{Escherichia coli} possess \textit{flagellae} driven by small motors which respond to gradients of signalling molecules in the environment. 
Chemotaxis also plays an important role in embryonal development, e.g. in the development of blood vessels (\textit{angiogenesis}), which is also a crucial step in tumour growth. 
Starting from the basic Keller-Segel model, many different model extensions are conceivable. 
A broad range of those is summarized in the review articles by Hillen and Painter \cite{hillen2009} and Horstmann \cite{horstmann2003}. 
Details on the modelling aspects can be found e.g.\ in the books by Murray \cite{murray2003} and Perthame \cite{perthame2007}.

In the model \eqref{eq:pde_u}\&\eqref{eq:pde_v} under consideration here,
$u$ is the time-dependent spatial density of the cells, and $v$ is the time-dependent concentration of the signalling substance.
Equation \eqref{eq:pde_u} describes the temporal change in cell density 
due to the directed drift of cells towards regions with higher concentration of the substance and due to undirected diffusion.
Equation \eqref{eq:pde_v} models the degradation of the signalling substance, as well as its production by the cells.
Two special aspects are included in this particular model:
nonlinear diffusion, i.e., the use of a non-constant, $u$-dependent mobility coefficent for the diffusive motion of the bacteria,
and signal-dependent chemotactic sensitivity, i.e., the use of the -- in general nonlinear -- response $\phi(v)$ instead of the concentration $v$ itself.
For the first, we refer to \cite{hillen2009} and the references therein for biological motivation. 
The second is motivated by the fact that the conversion of an external signal into a reaction of the considered microorganism (\textit{signal transduction}) 
often occurs by binding and dissociation of molecules to certain receptors. 
The movement of the cell is then caused rather by gradients in the number of receptors occupied by signalling molecules than by concentration gradients of signalling molecules itself. 
For growing concentrations, the number of bound receptors can exhibit a saturation, such that the gradient vanishes. 
In \cite{hillen2009,segel1977,lapidus1976}, this was included into the model by the \textit{chemotactic sensitivity function}
\begin{align*}
  \phi'(v)&=-\frac{1}{(1+v)^2},
\end{align*}
which fits into our model with the response function $\phi$ defined in \eqref{eq:paradigm}.
Finally, an external background potential $W$ is included in order to generate a spatial confinement of the bacterial population. 

For the dynamics of the signalling substance, we assume linear diffusion according to Fick's laws and degradation with a constant, exponential rate $\kappa$. 
The nonnegative term $-\eps u \phi'(v)$ models the production of signalling substance by the microorganisms;
here it is taken into account that the cells might be the less active in producing additional substance the higher its local concentration already is.
This is consistent with the models presented in \cite[Sect. 6]{horstmann2003}.

By definition, $u$ and $v$ are density/concentration functions and thus should be nonnegative.
Note that it is part of our results that for given nonnegative initial data (of sufficient regularity), there exists a weak solution that is nonnegative for all times $t>0$.
On a formal level, nonnegativity is an easy consequence of the particular structure of the system \eqref{eq:pde_u}\&\eqref{eq:pde_v}.

\subsection{Relation to the existing literature}
The rapidly growing mathematical literature about the Keller-Segel model and its manifold variants is devoted primarily to the dichotomy \emph{global existence versus finite-time blow-up} of (weak, possibly measure-valued) solutions,
but the long-time behaviour of global solutions has been intensively investigated as well.

Global existence and blow-up in the classical parabolic-parabolic Keller-Segel model,
which is \eqref{eq:pde_u}\&\eqref{eq:pde_v} with $\phi(v)=-v$, $W\equiv0$ and \emph{linear} diffusion,
has been thoroughly studied by Calvez and Corrias \cite{calvez2008} in space dimension $d=2$, 
and by Corrias and Perthame \cite{corrias2008} in higher space dimensions $d>2$, see also \cite{biler2011,kozono2009,mizoguchi2013,nagai2003,senba2006,sugiyama2006,yamada2011}.
Variants with nonlinear diffusion and drift have been studied for instance by Sugiyama \cite{sugiyama2006_2, sugiyama2007}.
The results from \cite{sugiyama2006_2} already indicate that in the model \eqref{eq:pde_u}\&\eqref{eq:pde_v} under consideration, 
blow-up \emph{never} occurs, in accordance with Theorem \ref{thm:existence}.

In the aforementioned works \cite{corrias2008,nagai2003}, the intermediate asymptotics of global solutions have been studied as well:
it is proved that the cell density converges to the self-similar solution of the heat equation at an algebraic rate,
i.e., in a properly scaled frame, the density approaches a Gaussian.
See also \cite{rosado2008} for an extension of this result to a model with size-exclusion.
Similar asymptotic behaviour has been proved in models with nonlinear, homogeneous diffusion,
e.g.\ by Sugiyama and Luckhaus \cite{luckhaus2006,luckhaus2007}.
There, the intermediate asymptotics are that of a porous medium equation with the respective homogeneous nonlinearity,
i.e., the rescaled bacterial density converges to a Barenblatt profile.
These intermediate asymptotics are -- at least morally -- related to Theorem \ref{thm:convergence}:
recall that algebraic convergence to self-similarity for the unconfined porous medium equation 
is comparable to exponential convergence to an equilibrium for the equation with $\lambda$-convex confinement.

The fully parabolic model \eqref{eq:pde_u}\&\eqref{eq:pde_v} with a non-linear response $\phi$ has not been rigorously analyzed so far,
with the following exception:
In her thesis \cite{post1999}, Post proves existence and uniqueness of solutions 
to a similar system with linear diffusion and vanishing confinement on a bounded domain by non-variational methods
and obtains convergence to the (spatially homogeneous) stationary solution from compactness arguments. Variants of the classical parabolic-parabolic or parabolic-elliptic Keller-Segel models with a nonlinear chemotactic sensitivity coefficient have also been studied e.g. in \cite{nagai1998,winkler2010}.

Despite the fact that energy/entropy methods are one of the key tools for the analysis of Keller-Segel-type systems, 
the use of genuine variational methods is relatively recent in that context.
The variational machinery of gradient flows in transportation metrics, 
originally developed by Jordan, Kinderlehrer and Otto \cite{jko1998} for the linear Fokker-Planck equation,
has been applied to a variety of dynamical systems:
mainly to nonlinear diffusion \cite{carrillo2000, otto2001, carrillo2006, agueh2008},
but also to aggregation \cite{carrillo2003kinetic,carrillo2006contractions,carrillo2011global} and fourth-order equations \cite{giacomelli2001, gianazza2009, matthes2009}.
For the parabolic-elliptic Keller-Segel model, which can be reduced to a single nonlocal scalar equation,
the variational framework was established by Blanchet, Calvez and Carrillo \cite{blanchet2008},
who represented the evolution as a gradient flow of an appropriate potential with respect to the Wasserstein distance 
and constructed a numerical scheme on these grounds.
Later, the gradient flow structure has been used for a detailed analysis of the basin of attraction in the critical mass case by Blanchet, Carlen and Carrillo \cite{carrillo2012} (see also e.g. \cite{blanchet2009,calvez_carrillo2012,lopezgomez2013}).

The parabolic-parabolic Keller-Segel model was somewhat harder to fit into the framework,
since the two equations are (formally) gradient flows with respect to \emph{different} metrics: Wasserstein and $L^2$.
The first rigorous analytical result on grounds of this structure was given by Blanchet and Lauren\c cot in \cite{blanchet2012},
where they constructed weak solutions for the system with critical exponents of nonlinear diffusion.
Later their result was generalized by the first author of this work to other, non-critical parameter situations in \cite{zinsl2012}. 
To the best of our knowledge, our approach taken here to prove long-time asymptotics by gradient flow techniques in a combined Wasserstein-$L^2$-metric is novel.

\subsection{Plan of the paper}
First, we summarize common facts and definitions on gradient flows in metric spaces in Section \ref{sec:pre}.
After that, various properties of the entropy functional are derived in Section \ref{sec:prop_H}.
On grounds of these properties, we construct a weak solution by means of the minimizing movement scheme in Section \ref{sec:existence},
proving Theorem \ref{thm:existence}.
Existence, uniqueness and regularity of stationary solutions is studied in Section \ref{sec:stat}, and the proof of Theorem \ref{thm:convergence} is completed in Section \ref{sec:conv}.

\section{Preliminaries}\label{sec:pre}
In this section, we recall the relevant definitions and properties related to gradient flows in metric spaces $(X,\metr)$, following \cite{savare2008}.
The two metric spaces of interest here are $L^2(\Rd)$ with the metric induced by the norm,
and the space $\prb(\Rd)$ of probability measures, endowed with the $L^2$-Wasserstein distance $\W_2$.
We also discuss the compound metric $\dist$ from \eqref{eq:dist}.

\subsection{Spaces of probability measures and the Wasserstein distance}\label{subsec:wasser}
We denote by $\prbb(\Rd)$ the space of probability measures on $\Rd$.
By abuse of notation, we will frequently identify absolutely continuous measures $\mu\in\prbb(\Rd)$
with their respective (Lebesgue) density functions $u=\dn\mu/\dn x\in L^1_+(\Rd)$, where $L^p_+(\Rd)$ for $p\ge 1$ denotes the subspace of those $L^p(\Rd)$ functions with nonnegative values.

A sequence $(\mu_n)_{n\in\N}$ in $\prbb(\Rd)$ is called \emph{narrowly convergent} to its limit $\mu\in\prbb(\Rd)$ if
\begin{align*}
 \lim_{n\to\infty}\int_{\Rd}\varphi(x)\dd\mu_n(x)&=\int_\Rd \varphi(x)\dd\mu(x).
\end{align*}
for every bounded, continuous function $\varphi:\,\Rd\to\R$.
By $\prb(\Rd)$, we denote the subspace of those $\mu\in\prbb(\Rd)$ with finite second moment
\begin{align*}
 \m_2(\mu)&:=\int_{\Rd}|x|^2\dd\mu(x).
\end{align*}
$\prb(\Rd)$ turns into a complete metric space when endowed with the $L^2$-Wasserstein distance $\W_2$.
We do not recall the general definition of $\W_2$ here.
Instead, since we are concerned with absolutely continuous measures in $\prb(\Rd)$ only,
we remark that for probability density functions $u_1,u_2\in L^1_+(\Rd)$,
the Wasserstein distance is given by the following infimum
\begin{align*}
 \W_2^2(u_1,u_2)&=\inf\left\{\int_{\Rd}|t(x)-x|^2u_1(x)\dd x\,\bigg.\bigg|\,t:\Rd\to\Rd \text{ Borel-measurable and }t\# u_1=u_2\right\},
\end{align*}
where $t\# u$ denotes the \emph{push-forward} w.r.t. the map $t$. In this case, the infimum above is attained
by an optimal transport map \cite[Thm. 2.32]{villani2003}.
Convergence in the metric space $(\prb(\Rd),\W_2)$ is equivalent to narrow convergence and convergence of the second moment.
Further, $\W_2$ is lower semicontinuous in both components with respect to narrow convergence.

\subsection{Geodesic convexity and gradient flows in metric spaces}\label{subsec:gf}
A functional $\auxil:X\to\R\cup\{\infty\}$ defined on the metric space $(X,\metr)$ is called \emph{geodesically $\lambda$-convex}
for some $\lambda\in\R$ if for every $w_0,w_1\in X$ and $s\in [0,1]$, one has
\begin{align*}
 \auxil(w_s)\le (1-s)\auxil(w_0)+s\auxil(w_1)-\frac{\lambda}{2}s(1-s)\metr^2(w_0,w_1),
\end{align*}
where $w_s:\,[0,1]\to X,\,s\mapsto w_s$ is a \emph{geodesic connecting} $w_0$ and $w_1$.

On $L^2(\Rd)$, the (unique up to rescaling) geodesic from $w_0$ to $w_1$ is given by linear interpolation, i.e., $w_s=(1-s)w_0+sw_1$.
Hence a functional $\mathfrak{F}:L^2(\Rd)\to\R\cup\{\infty\}$ of the form
\begin{align*}
 \mathfrak{F}(w) = \int_\Rd f\big(w(x),\dff w(x),\dff^2 w(x)\big)\dd x
\end{align*}
with a given continuous function $f:\R\times\R^d\times\R^{d\times d}\to\R$ is $\lambda$-convex
iff $(z,p,Q)\mapsto f(z,p,Q)-\frac{\lambda}{2} z^2$ is (jointly) convex.

In the metric space $(\prb(\Rd),\W_2)$, geodesic $\lambda$-convexity is a much more complicated concept.
We recall two important classes of $\lambda$-convex functionals (see e.g. \cite[Ch. 9.3]{savare2008}, {\cite[Thm. 5.15]{villani2003}}).

\begin{thm}[Criteria for geodesic convexity in $(\prb(\Rd),\W_2)$]\label{thm:crit_conv}
The following statements are true:
 \begin{enumerate}[(a)]
 \item Let a convex function $h\in C^0([0,\infty))$ be given, and define the functional $\mathfrak{H}(u):=\int_\Rd h(u(x))\dd x$ for $u\in\prb(\Rd)$.
   Then $\mathfrak{H}$ extends naturally to a lower semicontinuous functional on $\prb(\Rd)$.
   Now, if $h(0)=0$ and $r\mapsto r^dh(r^{-d})$ is convex and nonincreasing on $(0,\infty)$,
   then $\mathfrak{H}$ is $0$-geodesically convex in $(\prb(\Rd),\W_2)$.
 \item Let a function $W\in C^0(\Rd)$ be given, and define the functional $\mathfrak{H}(\mu):=\int_\Rd W(x)\dd\mu(x)$ for all $\mu\in\prb(\Rd)$.
   If $W$ is $\lambda$-convex (as a functional on the metric space $\Rd$ with the Euclidean distance) for some $\lambda\in\R$,
   then $\mathfrak{H}$ is $\lambda$-geodesically convex in $(\prb(\Rd),\W_2)$.
 \end{enumerate}
\end{thm}

%
%
Next, we introduce a notion of \emph{gradient flow}.
There are various possible characterizations.
For our purposes here, we need the following very strong one.
\begin{definition}
 Let $\auxil:X\to\R\cup\{\infty\}$ be a lower semicontinuous functional on the metric space $(X,\metr)$.
 A continuous semigroup $\flow{(\cdot)}{\auxil}$ on $(X,\metr)$ is called \emph{$\kappa$-flow} for some $\kappa\in\R$,
 if the \emph{evolution variational inequality}
 \begin{align}
   \label{eq:evi}
   \frac{1}{2}\frac{\dn^+}{\dn t}\metr^2(\flow{t}{\auxil}(w),\tilde w)+\frac{\kappa}{2}\metr^2(\flow{t}{\auxil}(w),\tilde w)+\auxil(\flow{t}{\auxil}(w))
   &\le\auxil(\tilde w),
\end{align}
 holds for arbitrary $w,\tilde w$ in the domain of $\auxil$, and for all $t\ge 0$.
\end{definition}
If $\flow{(\cdot)}{\auxil}$ is a $\kappa$-flow for the $\lambda$-convex functional $\auxil$,
then $\flow{(\cdot)}{\auxil}$ is also a \emph{gradient flow} for $\auxil$ in essentially all possible interpretations of that notion.
For the metric spaces $(\prb(\Rd),\W_2)$ and $L^2(\Rd)$,
it can be proved that every lower semicontinuous and geodesically $\lambda$-convex functional possesses a unique $\kappa$-flow, with $\kappa:=\lambda$
(see \cite[Theorem 11.1.4]{savare2008} and \cite[Corollary 4.3.3]{savare2008}, respectively).

In these metric spaces, $\lambda$-geodesic convexity with $\lambda>0$ implies existence and uniqueness of a minimizer $w_{\min}$ of $\auxil$, for which the following holds (see e.g. \cite[Lemma 2.4.8, Thm. 4.0.4]{savare2008}):
\begin{align}
\frac{\lambda}{2}\metr^2(w,w_{\min})&\le \auxil(w)-\auxil(w_{\min})\le \frac{1}{2\lambda}\lim_{h\downarrow 0}\frac{\auxil(w)-\auxil(\flow{h}{\auxil}(w))}{h}.\label{eq:subdiff}
\end{align}
\begin{remark}[Formal calculation of evolution equations associated to gradient flows]\label{rem:formal}
 In the  metric spaces of interest here,
 one can explicitly write an evolution equation for the flow $\flow{(\cdot)}{\auxil}$ of a sufficiently regular functional $\auxil$,
 see e.g.\ \cite[Sect. 8.2]{villani2003}.
 On $(\prb(\Rd),\W_2)$, one has
 \begin{align*}
   \partial_t \flow{t}{\auxil}(w) &=\mathrm{div}\left(\flow{t}{\auxil}(w) \dff \left(\frac{\delta\auxil}{\delta w}(\flow{t}{\auxil}(w))\right)\right),
 \end{align*}
 and on $L^2(\Rd)$, one has
 \begin{align*}
   \partial_t \flow{t}{\auxil}(w)&=-\frac{\delta\auxil}{\delta w}(\flow{t}{\auxil}(w)).
 \end{align*}
 Here, $\frac{\delta\auxil}{\delta w}$ stands for the usual first variation of the functional $\auxil$ on $L^2$.
\end{remark}

\subsection{The metric $\dist$}
It is easily verified that $\theX:=\prb(\Rdrei)\times L^2_+(\Rdrei)$ becomes a complete metric space when endowed with the compound metric $\dist$ defined in \eqref{eq:dist}.
The topology on $\theX$ induced by $\dist$ is that of the cartesian product.
Moreover:
\begin{lemma}
 \label{lem:lsc_dist}
 The distance $\dist$ is weakly lower semicontinuous on $\theX$ in the following sense:
 If $(u_n,v_n)_{n\in\N}$ is a sequence in $\theX$ such that $u_n$ converges to $u\in\prb(\Rdrei)$ narrowly and $v_n$ converges to $v\in L^2(\Rdrei)$ weakly in $L^2(\Rdrei)$,
 then
 \begin{align*}
   \dist((u,v),(\tilde u,\tilde v)) \le \liminf_{n\to\infty}\dist((u_n,v_n),(\tilde u,\tilde v))
 \end{align*}
 holds, for every $(\tilde u,\tilde v)\in\theX$.
\end{lemma}
For our purposes, it suffices to discuss convexity and gradient flows for functionals $\Phi:\theX\to\R\cup\{\infty\}$ of the separable form $\Phi(u,v)=\Phi_1(u)+\Phi_2(v)$.
One immediately verifies
\begin{lemma}
 Assume that $\Phi_1$ and $\Phi_2$ are $\lambda$-convex and lower semicontinuous functionals on the respective spaces $(\prb(\Rdrei),\W_2)$ and $L^2(\Rdrei)$,
 and denote their respective gradient flows by $\flow{(\cdot)}{1}$ and $\flow{(\cdot)}{2}$.
 Then $\Phi:\theX\to\R\cup\{\infty\}$ with $\Phi(u,v)=\Phi_1(u)+\Phi_2(v)$ is a $\lambda$-convex and lower semicontinuous functional on $(\theX,\dist)$,
 and the semigroup $\flow{(\cdot)}{\Phi}$ given by $\flow{t}{\Phi}(u,v)=(\flow{t}{1}(u),\flow{t}{2}(v))$ is a $\lambda$-flow for $\Phi$.
\end{lemma}

\section{Properties of the entropy functional}\label{sec:prop_H}
Recall the definition of the metric space $(\theX,\dist)$.
We define the \emph{entropy functional} $\calH:\theX\to\R\cup\{\infty\}$ as follows.
For all $(u,v)\in \theX\cap(L^2(\Rdrei)\times W^{1,2}(\Rdrei))$, set
\begin{align}
 \label{eq:H}
 \calH(u,v) :=\int_\Rdrei \left(\frac{1}{2}u^2+uW+\frac{1}{2}|\dff v|^2+\frac{\kappa}{2}v^2+\eps u\phi(v)\right)\dd x,
\end{align}
which is a finite value by our assumptions on $\phi$ and $W$. For all other $(u,v)\in\prb(\Rdrei)\times L^2(\Rdrei)$, we set $\calH(u,v)=+\infty$.
\begin{prop}[Properties of the entropy functional $\calH$]\label{prop:propH}
 The functional $\calH$ defined in \eqref{eq:H} has the following properties:
 \begin{enumerate}[(a)]
 \item There exist $C_0,\,C_1>0$ such that
\begin{align}
\label{eq:Hest}
\calH(u,v)&\ge C_0\left[\|u\|_{L^2(\Rdrei)}^2+\m_2(u)+\|v\|_{W^{1,2}(\Rdrei)}^2-C_1\right].
\end{align}
In particular, $\calH$ is bounded from below.
 \item $\calH$ is weakly lower semicontinuous in the following sense:
   For every sequence $(u_n,v_n)_{n\in\N}$ in $\theX$,
   where $(u_n)_{n\in\N}$ converges narrowly to some $u\in\prb(\Rdrei)$ and where $(v_n)_{n\in\N}$ converges weakly in $L^2(\Rdrei)$ to some $v\in  L^2(\Rdrei)$,
   one has
   \begin{align*}
     \calH(u,v)\le\liminf_{n\to\infty}\calH(u_n,v_n).
   \end{align*}
 \item For sufficiently small $\eps>0$, $\calH$ is $\lambda'$-geodesically convex for some $\lambda'>0$ with respect to the distance $\dist'((u_1,v_1),(u_2,v_2)):=\sqrt{\|u_1-u_2\|_{L^2(\Rdrei)}^2+\|v_1-v_2\|_{L^2(\Rdrei)}^2}$.
 \end{enumerate}
\end{prop}
\begin{proof}
 For part (a), we observe that due to $\lambda_0$-convexity of $W$, one has $W(x)\ge \frac{\lambda_0}{4}|x|^2-\frac{\lambda_0}{2} |x_{\min}|^2$, where $x_{\min}\in\Rdrei$ is the unique minimizer of $W$. Moreover, with convexity of $\phi$, we deduce
\begin{align*}
\int_\Rdrei u\phi(v)\dd x&\ge \phi(0)+\phi'(0)\|uv\|_{L^1(\Rdrei)}\ge \phi(0)+C\phi'(0) \|\dff v\|_{L^2(\Rdrei)}\|u\|_{L^2(\Rdrei)}^{1/3},
\end{align*}
using that $\|u\|_{L^1(\Rdrei)}=1$ and the following chain of inequalities:
\begin{align}
\label{eq:holsob}
\|uv\|_{L^1(\Rdrei)}&\le \|u\|_{L^{6/5}(\Rdrei)}\|v\|_{L^6(\Rdrei)}\le C\|\dff v\|_{L^2(\Rdrei)}\|u\|_{L^1(\Rdrei)}^{2/3}\|u\|_{L^2(\Rdrei)}^{1/3}.
\end{align}

 All in all, we arrive at
\begin{align*}
\calH(u,v)&\ge \frac12 \|u\|_{L^2(\Rdrei)}^2+\frac{\lambda_0}{4}\m_2(u)-\lambda_0|x_{\min}|^2+\frac12\|\dff v\|_{L^2(\Rdrei)}^2+\frac{\kappa}{2}\|v\|_{L^2(\Rdrei)}^2\\&-\eps|\phi(0)|-\eps C|\phi'(0)|\|\dff v\|_{L^2(\Rdrei)}\|u\|_{L^2(\Rdrei)}^{1/3}.
\end{align*}
From this, the desired estimate follows by means of Young's inequality.

 In (b), the claimed lower semicontinuity of the integral with $\eps=0$ follows from joint convexity of the map
 \begin{align*}
   \R_+\times\R_+\times\Rdrei\ni(r,z,p)\mapsto \frac12r^2+W(x)r+\frac12|p|^2+\frac\kappa2z^2,
 \end{align*}
 for every $x\in\Rdrei$.
 It thus remains to prove semicontinuity of the integral of $u\phi(v)$.
 Let a sequence $(u_n,v_n)_{n\in\N}$ with the mentioned properties be given,
 and assume -- without loss of generality -- that $\calH(u_n,v_n)\to H<\infty$.
 It then follows by \eqref{eq:Hest} that $(u_n)_{n\in\N}$ and $(v_n)_{n\in\N}$ are bounded sequences in $L^2(\Rdrei)$ and in $W^{1,2}(\Rdrei)$, respectively. Moreover, the sequence of second moments $(\m_2(u_n))_{n\in\N}$ is bounded.
 Hence, $(u_n)_{n\in\N}$ converges to $u$ weakly in $L^2(\Rdrei)$,
 and $(v_n)_{n\in\N}$ converges to $v$ weakly in $W^{1,2}(\Rdrei)$ and strongly in $L^2(B_R(0))$, for every ball $B_R(0)\subset\Rdrei$.
 Recalling our assumptions \eqref{eq:assumpphi} on $\phi$,
 we conclude that
 \begin{align*}
   |\phi(v_n)-\phi(v)|^2\le\phi'(0)^2|v_n-v|^2,
 \end{align*}
 and thus $(\phi(v_n))_{n\in\N}$ converges to $\phi(v)$ strongly in $L^2(B_R(0))$. We proceed by a truncation argument. Let therefore $R>0$ and choose $\beta_R\in C^\infty(\R^3)$ with
\begin{align*}
0\le\beta_R\le 1,\quad \beta_R\equiv 1 \text{ on }B_R(0),\quad \beta_R\equiv 0\text{ on }\Rdrei\backslash B_{2R}(0).
\end{align*}
Using the triangle inequality, we see
\begin{align}
\label{eq:lsctriangle}
&\left|\int_\Rdrei (u_n\phi(v_n)-u\phi(v))\dd x\right|\nonumber\\&\le \left|\int_\Rdrei \phi(v)(u_n-u)\dd x\right|+\left|\int_\Rdrei \beta_Ru_n(\phi(v_n)-\phi(v))\dd x\right|+\left|\int_\Rdrei (1-\beta_R)u_n(\phi(v_n)-\phi(v))\dd x\right|.
\end{align}
Since $u_n\rightharpoonup u$ weakly on $L^2(\Rdrei)$ and $\phi(v)\in L^2(\Rdrei)$, the first term in \eqref{eq:lsctriangle} converges to zero. The same holds for the second one due to strong convergence of $\phi(v_n)$ to $\phi(v)$ on $L^2(B_{2R}(0))$ and boundedness of $\|u_n\|_{L^2(\Rdrei)}$. The third term in \eqref{eq:lsctriangle} can be estimated using \eqref{eq:holsob}:
\begin{align*}
\left|\int_\Rdrei (1-\beta_R)u_n(\phi(v_n)-\phi(v))\dd x\right|&\le \|\phi(v_n)-\phi(v)\|_{L^6(\Rdrei)}\|u_n\|_{L^{6/5}(\Rdrei\backslash B_R(0))},
\end{align*}
and consequently
\begin{align*}
\left|\int_\Rdrei (1-\beta_R)u_n(\phi(v_n)-\phi(v))\dd x\right|&\le C\|\phi(v_n)-\phi(v)\|_{W^{1,2}(\Rdrei)}\|u_n\|_{L^2(\Rdrei)}^{1/3}\left(\int_{\Rdrei\backslash B_R(0)}\frac{|x|^2}{R^2}u_n(x)\dd x\right)^{2/3}\\
&\le CR^{-4/3} (\|\phi(v_n)\|_{W^{1,2}(\Rdrei)}+\|\phi(v)\|_{W^{1,2}(\Rdrei)})\|u_n\|_{L^2(\Rdrei)}^{1/3}(\m_2(u_n))^{2/3}\\&\le 2\tilde CR^{-4/3}.
\end{align*}
Hence, the following holds for all $R>0$:
\begin{align*}
\limsup_{n\to\infty}\left|\int_\Rdrei (u_n\phi(v_n)-u\phi(v))\dd x\right|&\le 2\tilde CR^{-4/3},
\end{align*}
proving the claim.

 Finally, to prove (c), consider a geodesic $w_s=(u_s,v_s)$ with respect to the flat metric $\dist'$,
 that is $u_s=(1-s)u_0+su_1$ and $v_s=(1-s)v_0+sv_1$ for given $u_0,u_1\in(\prb\cap L^2)(\Rdrei)$ and $v_0,v_1\in W^{1,2}(\Rdrei)$.
 It then follows that
 \begin{align*}
   \frac{\dn^2}{\dn s^2}\calH(u_s,v_s)
   &=\int_\Rdrei\big((u_1-u_0)^2+|\dff(v_1-v_0)|^2+\kappa(v_1-v_0)^2 \\
   & \qquad +2\eps \phi'(v_s) (u_1-u_0)(v_1-v_0)+\eps u_s\phi''(v_s)(v_1-v_0)^2\big)\dd x \\
   &\ge \int_\Rdrei
   \begin{pmatrix} u_1-u_0 \\ v_1-v_0 \end{pmatrix}^\mathrm{T}A_s\begin{pmatrix} u_1-u_0 \\ v_1-v_0 \end{pmatrix} \dd x
   \quad \text{with} \quad A_s := \begin{pmatrix} 1 & \eps\phi'(v_s) \\ \eps\phi'(v_s) & \kappa \end{pmatrix},
 \end{align*}
 where we have used that $\phi$ is convex.
 Thus, $\calH$ is $\lambda'$-convex with respect to $\dist'$ if $A_s\ge\lambda'\eins$ for all $s\in[0,1]$.
 Recalling that $0<-\phi'(v_s)\le-\phi'(0)$ by hypothesis \eqref{eq:assumpphi},
 it follows from elementary linear algebra that $\eps^2\phi'(0)^2<\kappa$ is sufficient to find a suitable $\lambda'>0$ with $A_s\ge\lambda'\eins$.
\end{proof}

\section{Existence of weak solutions}\label{sec:existence}
In this section, we prove Theorem \ref{thm:existence} by construction of a weak solution using the minimizing movement scheme.

\subsection{Time discretization}\label{subsec:discr_constr}
Recall the discretization scheme from \eqref{eq:jko}. We introduce the step size $\tau>0$ and define the associated Yoshida penalization $\calH_\tau$ of the entropy by
\begin{align}
 \label{eq:yoshida}
 \calH_\tau(u,v\,|\,\tilde u,\tilde v):=\frac{1}{2\tau}\dist^2((u,v),(\tilde u,\tilde v))+\calH(u,v)
\end{align}
for all $(u,v),(\tilde u,\tilde v)\in\theX$.
Set $(u_\tau^0,v_\tau^0):=(u_0,v_0)$ and define the sequence $(u_\tau^n,v_\tau^n)_{n\in\N}$ inductively by choosing
\begin{align}
(u_\tau^n,v_\tau^n)\in\operatorname*{argmin}_{(u,v)\in\theX}\calH_\tau(u,v\,|\,u_\tau^{n-1},v_\tau^{n-1}).\label{eq:minmov}
\end{align}
\begin{lemma}
 \label{prop:minmov}
 For every $(\tilde u,\tilde v)\in \theX$, there exists at least one minimizer $(u,v)\in\theX$ of $\calH_\tau(\cdot|\,\tilde u,\tilde v)$
 that satisfies $u\in L^2(\Rdrei)$ and $v\in W^{1,2}(\Rdrei)$.
\end{lemma}
\begin{proof}
 The proof is an application of the direct methods from the calculus of variations to the functional $\calH_\tau(\cdot|\,\tilde u,\tilde v)$.

 First, observe that on any given sublevel $S$ of $\calH_\tau(\cdot|\,\tilde u,\tilde v)$,
 both $\W_2(u,\tilde u)$ and $\|v\|_{L^2(\Rdrei)}$ are uniformly bounded.
 The first bound implies that also the second moment $\m_2(u)$ is uniformly bounded,
 and thus the $u$-components in $S$ belong to a subset of $\prb(\Rdrei)$ that is relatively compact in the narrow topology by Prokhorov's theorem.
 The other bound implies via Alaoglu's theorem that the $v$-components belong to a weakly relatively compact subset of $L^2(\Rdrei)$.

 Next, recall the properties of $\calH$ and of $\dist$ given in Proposition \ref{prop:propH} and Lemma \ref{lem:lsc_dist}.
 From these, it follows that  $\calH_\tau(\cdot|\,\tilde u,\tilde v)$ is lower semicontinuous
 with respect to narrow convergence in the first and $L^2$-weak convergence in the second component.

 Combining these properties with the fact that $\calH_\tau(\cdot|\,\tilde u,\tilde v)$ is bounded from below (e.g.\ by zero),
 the existence of a minimizer follows.
 The additional regularity is a consequence of the fact that the proper domain of $\calH$ is a subset of $L^2(\Rdrei)\times W^{1,2}(\Rdrei)$.
\end{proof}

Given the sequence $(u_\tau^n,v_\tau^n)_{n\in\N}$, define the \emph{discrete solution} $(u_\tau,v_\tau):[0,\infty)\to\theX$ as in \eqref{eq:interpol} by piecewise constant interpolation:
\begin{align}
 (u_\tau,v_\tau)(t)&:=(u_\tau^n,v_\tau^n)\text{ for }t\in ((n-1)\tau,n\tau]\text{ and }n\ge 1.\label{eq:disc_sol}
\end{align}

We start be recalling a collection of estimates on $(u_\tau,v_\tau)$ that follows immediately from the construction by minimizing movements.
\begin{prop}[Classical estimates]\label{prop:class}
 The following holds for $T>0$:
 \begin{align}
   \calH(u_\tau^n,v_\tau^n)&\le \calH(u_0,v_0)<\infty\qquad\forall n\ge 0,\label{eq:apriori1}\\
   \sum_{n=1}^\infty \W_2^2(u_\tau^n,u_\tau^{n-1})&\le 2\tau(\calH(u_0,v_0)-\inf\calH),\label{eq:apriori2}\\
\W_2(u_\tau(s),u_\tau(t))&\le\left[2(\calH(u_0,v_0)-\inf\calH)\max(\tau,|t-s|)\right]^{1/2}\qquad\forall 0\le s,t\le T,\label{eq:apriori3}\\
   \sum_{n=1}^\infty \|v_\tau^n-v_\tau^{n-1}\|_{L^2(\Rdrei)}^2&\le 2\tau(\calH(u_0,v_0)-\inf\calH),\label{eq:apriori4}\\
\|v_\tau(s)-v_\tau(t)\|_{L^2(\Rdrei)}&\le\left[2(\calH(u_0,v_0)-\inf\calH)\max(\tau,|t-s|)\right]^{1/2}\qquad\forall 0\le s,t\le T,\label{eq:apriori5}
 \end{align}
the infimum $\inf\calH$ of $\calH$ on $\theX$ being finite.
\end{prop}

By the well-known \emph{JKO method} \cite{jko1998}, we derive an approximate weak fomulation satisfied by $(u_\tau,v_\tau)$.
The idea is to choose test functions $\eta,\gamma\in C^\infty_c(\Rdrei)$
and perturb the minimizer $(u_\tau^n,v_\tau^n)$ of the functional $\calH_\tau(\,\cdot\,|\,u_\tau^{n-1},v_\tau^{n-1})$
over an auxiliary time $s\ge0$ as follows:
\begin{align*}
 u_\tau^n \rightsquigarrow \flow{s}{{\dff\eta}}\#u_\tau^n, \quad v_\tau^n \rightsquigarrow v+s\gamma.
\end{align*}
Here $\flow{{(\cdot)}}{{\dff\eta}}$ is the flow on $\Rdrei$ generated by the gradient vector field $\dff\eta$.
Since the calculations are very similar to the ones performed in \cite{zinsl2012}, we skip the details and directly state the result.
\begin{lemma}
 For all $n\in\N$ and all test functions $\eta,\gamma\in C^\infty_c(\Rdrei)$ and $\psi\in C^\infty_c((0,\infty))\cap C([0,\infty))$, the following \emph{discrete weak formulation} holds:
 \begin{align}
   \label{eq:discrete_weak}
   \begin{split}
0=\int_0^\infty\int_\Rdrei&\left[u_\tau(t,x)\eta(x)-v_\tau(t,x)\gamma(x)\right]
     \frac{\psi\left(\left\lfloor \frac{t}{\tau}\right\rfloor\tau\right)-\psi\left(\left\lfloor\frac{t}{\tau}\right\rfloor\tau+\tau\right)}{\tau}\dd x\dd t
     +O(\tau)\\
     +\int_0^\infty\int_\Rdrei &\psi\left(\left\lfloor \frac{t}{\tau}\right\rfloor\tau\right)\bigg(-\frac{1}{2}u_\tau(t,x)^2\Delta\eta(x)+u_\tau(t,x)\dff W(x)\cdot\dff\eta(x) \\
     &\qquad +\dff v_\tau(t,x)\cdot\dff\gamma(x)+\kappa v_\tau(t,x)\gamma(x)\\
     &\qquad +\eps u_\tau(t,x)\phi'(v_\tau(t,x))[\gamma(x)+\dff v_\tau(t,x)\cdot\dff\eta(x)]\bigg)\dd x\dd t.
   \end{split}
\end{align}
\end{lemma}
Our goal for the rest of this section is to pass to the limit $\tau\downarrow0$ in \eqref{eq:discrete_weak} and obtain the (time-continuous) weak formulation \eqref{eq:pde_weak1}\&\eqref{eq:pde_weak2}.

\subsection{Regularity of the discrete solution}
Since the discrete weak formulation \eqref{eq:discrete_weak} contains nonlinear terms with respect to $u_\tau$ and $v_\tau$,
further compactness estimates are needed to pass to the continuous time limit $\tau\to 0$. As a preparation, we state the following lemma:

\begin{lemma}[Flow interchange lemma {\cite[Thm. 3.2]{matthes2009}}]
 \label{lem:flowinterchange}
 Let $\auxil$ be a proper, lower semicontinuous and $\lambda$-geodesically convex functional on $(\theX,\dist)$,
 which is defined at least on $\theX\cap L^2(\Rdrei)\times W^{1,2}(\Rdrei)$.
 Further, assume that $\flow{(\cdot)}{\auxil}$ is a $\lambda$-flow for $\auxil$.
 Then, the following holds for every $n\in\N$:
 \begin{align*}
   \auxil(u_\tau^n,v_\tau^n)+\tau \mathrm{D}^\auxil\calH(u_\tau^n,v_\tau^n)+\frac{\lambda}{2}\dist^2((u_\tau^n,v_\tau^n),(u_\tau^{n-1},v_\tau^{n-1}))&\le \auxil(u_\tau^{n-1},v_\tau^{n-1}).
 \end{align*}
 There, $\mathrm{D}^\auxil\calH(w)$ denotes the \textit{dissipation} of the entropy $\calH$ along $\flow{(\cdot)}{\auxil}$,
 i.e.
 \begin{align*}
   \mathrm{D}^\auxil\calH(w):=\limsup_{h\downarrow 0}\frac{\calH(w)-\calH(\flow{h}{\auxil}(w))}{h}.
 \end{align*}
\end{lemma}

The necessary additional regularity is provided by the following estimate on the minimizers of $\calH_\tau$:
\begin{prop}[Additional regularity]\label{prop:reg_min}
 Let $(u,v),(\tilde u,\tilde v)\in\theX\cap(L^2(\Rdrei)\times W^{1,2}(\Rdrei))$ with \\$(u,v)\in\mathrm{argmin }\calH_\tau(\,\cdot\,|\,\tilde u,\tilde v)$. Denoting $\mathcal{E}(u):=\int_\Rdrei u\log(u)\dd x$ and $\mathcal{F}(v):=\int_\Rdrei \left(\frac{1}{2}|\dff v|^2+\frac{\kappa}{2}v^2\right)\dd x$, the following estimate holds for some constant $K>0$:
 \begin{align}
   &\|\dff u\|_{L^2(\Rdrei)}^2+\|\Delta v-\kappa v\|_{L^2(\Rdrei)}^2\nonumber\\&\le K\left(\|u\|_{L^2(\Rdrei)}^{2}+\|v\|_{W^{1,2}(\Rdrei)}^{2}+\|\Delta W\|_{L^\infty(\Rdrei)}+\frac{1}{\tau}\left(\mathcal{E}(\tilde u)-\mathcal{E}(u)+\mathcal{F}(\tilde v)-\mathcal{F}(v)\right)\right).\label{eq:reg_min}
 \end{align}
\end{prop}

\begin{proof}
The method of proof used here is based on the flow interchange lemma (Lemma \ref{lem:flowinterchange}). The idea is to calculate the dissipation of $\calH$ along the gradient flow of an auxiliary functional, namely the heat flow and the heat flow with decay, respectively.

Therefore, we recall that the functional $\mathcal{E}(u):=\int_\Rdrei u\log(u)\dd x$ is $0$-geodesically convex on $\mathscr{P}_2(\Rdrei)$ and its gradient flow $\flow{(\cdot)}{\mathcal{E}}$ is the heat flow satisfying
\begin{align*}
\partial_s \flow{s}{\mathcal{E}}(u)&=\Delta \flow{s}{\mathcal{E}}(u).
\end{align*}

Moreover, with the evolution variational inequality \eqref{eq:evi}, we deduce as in \cite{blanchet2012, zinsl2012} by integration over time using that $\mathcal{E}$ is a Lyapunov functional along $\flow{(\cdot)}{\mathcal{E}}$:
\begin{align}
\frac{1}{2}\left(\W_2^2(\flow{s}{\mathcal{E}}(u),\tilde u)-\W_2^2(u,\tilde u)\right)&\le\int_0^s(\mathcal{E}(\tilde u)-\mathcal{E}(\flow{\sigma}{\mathcal{E}}(u)))\dd \sigma\le s[\mathcal{E}(\tilde u)-\mathcal{E}(\flow{s}{\mathcal{E}}(u))].\label{eq:heat_evi}
\end{align}

Analogous to that, $\mathcal{F}(v):=\int_\Rdrei \left(\frac{1}{2}|\dff v|^2+\frac{\kappa}{2}v^2\right)\dd x$ is $\kappa$-geodesically convex on $L^2(\Rdrei)$ and its gradient flow $\flow{(\cdot)}{\mathcal{F}}$ is given by
\begin{align*}
\partial_s \flow{s}{\mathcal{F}}(v)&=\Delta \flow{s}{\mathcal{F}}(v)-\kappa \flow{s}{\mathcal{F}}(v).
\end{align*}

The application of the evolution variational inequality \eqref{eq:evi} then shows
\begin{align}
\frac{1}{2}\left(\|\flow{s}{\mathcal{F}}(v)-\tilde v\|_{L^2(\Rdrei)}^2-\|v-\tilde v\|_{L^2(\Rdrei)}^2\right)&\le\int_0^s(\mathcal{F}(\tilde v)-\mathcal{F}(\flow{\sigma}{\mathcal{F}}(v)))\dd \sigma\le s[\mathcal{F}(\tilde v)-\mathcal{E}(\flow{s}{\mathcal{F}}(v))].\label{eq:heat_dec_evi}
\end{align}

Well-known results of parabolic theory ensure that $(\flow{s}{\mathcal{E}}(u),\flow{s}{\mathcal{F}}(v))\in \theX\cap(L^2(\Rdrei)\times W^{1,2}(\Rdrei))$ if $(u,v)\in \theX\cap(L^2(\Rdrei)\times W^{1,2}(\Rdrei))$. For the sake of clarity, we introduce the notation $(\mathcal{U}_s,\mathcal{V}_s):=(\flow{s}{\mathcal{E}}(u),\flow{s}{\mathcal{F}}(v))$ and calculate for $s>0$:
\begin{align*}
&\frac{\dd}{\dd s}\mathcal{H}(\mathcal{U}_s,\mathcal{V}_s)\nonumber\\
&=\int_\Rdrei\bigg(\left[\mathcal{U}_s+W+\eps\phi(\mathcal{V}_s)\right]\Delta \mathcal{U}_s+\left[-\Delta\mathcal{V}_s+\kappa\mathcal{V}_s+\eps\mathcal{U}_s\phi'(\mathcal{V}_s)\right][\Delta\mathcal{V}_s-\kappa \mathcal{V}_s]\bigg)\dd x\nonumber\\
&=\int_\Rdrei\bigg(-|\dff \mathcal{U}_s|^2-\mathcal{U}_s\Delta W-(\Delta \mathcal{V}_s-\kappa \mathcal{V}_s)^2-\eps\phi'(\mathcal{V}_s)\dff \mathcal{V}_s\cdot \dff \mathcal{U}_s+\eps \mathcal{U}_s\phi'(\mathcal{V}_s)[\Delta\mathcal{V}_s-\kappa \mathcal{V}_s]\bigg)\dd x,
\end{align*}
where the last line follows by integration by parts. An application of Young's inequality yields
\begin{align*}
&\int_\Rdrei\bigg(-|\dff \mathcal{U}_s|^2-\mathcal{U}_s\Delta W-(\Delta \mathcal{V}_s-\kappa \mathcal{V}_s)^2-\eps\phi'(\mathcal{V}_s)\dff \mathcal{V}_s\cdot \dff \mathcal{U}_s+\eps \mathcal{U}_s\phi'(\mathcal{V}_s)[\Delta\mathcal{V}_s-\kappa \mathcal{V}_s]\bigg)\dd x\nonumber\\
&\le \int_\Rdrei\bigg(-\frac{1}{2}|\dff \mathcal{U}_s|^2-\mathcal{U}_s\Delta W-\frac{1}{2}(\Delta \mathcal{V}_s-\kappa \mathcal{V}_s)^2+\frac{1}{2}\eps^2\phi'(0)^2 (|\dff \mathcal{V}_s|^2+\mathcal{U}_s^2)\bigg)\dd x.
\end{align*}
Exploiting the monotonicity of the $L^2$ norm along $\flow{(\cdot)}{\mathcal{E}}$ and of the $W^{1,2}$ norm along $\flow{(\cdot)}{\mathcal{F}}$, one gets
\begin{align*}
&\int_\Rdrei\bigg(-\frac{1}{2}|\dff \mathcal{U}_s|^2-\mathcal{U}_s\Delta W-\frac{1}{2}(\Delta \mathcal{V}_s-\kappa \mathcal{V}_s)^2+\frac{1}{2}\eps^2\phi'(0)^2 (|\dff \mathcal{V}_s|^2+\mathcal{U}_s^2)\bigg)\dd x\nonumber\\
&\le-\frac{1}{2}\|\dff \mathcal{U}_s\|_{L^2(\Rdrei)}^2-\frac{1}{2}\|\Delta \mathcal{V}_s-\kappa \mathcal{V}_s\|_{L^2(\Rdrei)}^2\\&+\|\Delta W\|_{L^\infty(\Rdrei)}+\frac{1}{2}\eps^2\phi'(0)^2 (\|\dff v\|_{L^2(\Rdrei)}^2+\|u\|_{L^2(\Rdrei)}^2).
\end{align*}

All in all, we have estimated the dissipation of $\calH$ along $\flow{(\cdot)}{\mathcal{E}}$ and $\flow{(\cdot)}{\mathcal{F}}$:
\begin{align}
&\frac{\dd}{\dd s}\mathcal{H}(\mathcal{U}_s,\mathcal{V}_s)\le-\frac{1}{2}\left(\|\dff \mathcal{U}_s\|_2^2+\|\Delta \mathcal{V}_s-\kappa \mathcal{V}_s\|_2^2\right)+C\|u\|_2^{2}+C\|v\|_{W^{1,2}(\Rdrei)}^{2}+\|\Delta W\|_\infty.\label{eq:estR}
\end{align}

In the final step of the proof of Proposition \ref{prop:reg_min}, we use the minimizing property of $(u,v)$. Clearly,
\begin{align*}
0&\le\calH_\tau(\mathcal{U}_s,\mathcal{V}_s\,|\,\tilde u,\tilde v)-\calH_\tau(u,v\,|\,\tilde u,\tilde v).
\end{align*}
We insert \eqref{eq:heat_evi}, \eqref{eq:heat_dec_evi} and \eqref{eq:estR} and obtain
\begin{align*}
&\frac{1}{s}\int_0^s\left(\|\dff \mathcal{U}_\sigma\|_{L^2(\Rdrei)}^2+\|\Delta \mathcal{V}_\sigma-\kappa \mathcal{V}_\sigma\|_{L^2(\Rdrei)}^2\right)\dd \sigma\nonumber\\
&\le K\left(\|u\|_{L^2(\Rdrei)}^{2}+\|v\|_{W^{1,2}(\Rdrei)}^{2}+\|\Delta W\|_{L^\infty(\Rdrei)}+\frac{1}{\tau}\left(\mathcal{E}(\tilde u)-\mathcal{E}(\mathcal{U}_s)+\mathcal{F}(\tilde v)-\mathcal{F}(\mathcal{V}_s)\right)\right),
\end{align*}
for some constant $K>0$.
Similar to \cite{blanchet2012,zinsl2012,blanchet2013r}, passing to the $\liminf$ as $s\to 0$ yields \eqref{eq:reg_min} by lower semicontinuity of norms and continuity of the entropies $\mathcal{E}$ and $\mathcal{F}$ along their respective gradient flows.
\end{proof}

\subsection{Compactness estimates and passage to continuous time}\label{subsec:comp}
The following compactness estimates in addition to the results of Proposition \ref{prop:class} are needed to pass to the limit $\tau\to 0$ in the nonlinear terms of the discrete weak formulation \eqref{eq:discrete_weak} afterwards. The method of proof is essentially the same as in \cite[Sect. 7]{zinsl2012}. For the sake of brevity, the details are omitted here.

\begin{prop}[Additional \textit{a priori} estimates]\label{prop:apriori}
Let $(u_\tau,v_\tau)$ be the discrete solution obtained by the minimizing movement scheme \eqref{eq:minmov}. Then the following holds for $T>0$:
\begin{align}
\m_2(u_\tau^n)&\le C_1<\infty\qquad\forall n\le\left\lfloor\frac{T}{\tau}\right\rfloor,\label{eq:apriori6}\\
\|u_\tau^n\|_{L^2(\Rdrei)}&\le C_3<\infty\qquad\forall n\ge 0,\label{eq:apriori7}\\
\|v_\tau^n\|_{W^{1,2}(\Rdrei)}&\le C_5<\infty \qquad\forall n\ge 0,\label{eq:apriori8}\\
\int_0^T\|u_\tau(t)\|_{W^{1,2}(\Rdrei)}^2\dd t&\le C_6<\infty,\label{eq:apriori9}\\
\int_0^T\|v_\tau(t)\|_{W^{2,2}(\Rdrei)}^2\dd t&\le C_7<\infty,\label{eq:apriori10}
\end{align}
with constants $C_j>0$ only depending on $T$ and the initial condition $(u_0,v_0)$.
\end{prop}

The estimates of Propositions \ref{prop:class} and \ref{prop:apriori} enable us to prove the existence of the continuous-time limit of the discrete solution:

\begin{prop}[Continuous time limit]\label{prop:conv_disc}
Let $(\tau_k)_{k\ge 0}$ be a vanishing sequence of step sizes, i.e. $\tau_k\to 0$ as $k\to\infty$, and let $(u_{\tau_k},v_{\tau_k})_{k\ge 0}$ be the corresponding sequence of discrete solutions obtained by the minimizing movement scheme.\\ Then there exist a subsequence (non-relabelled) and limit curves $u\in C^{1/2}([0,T],\mathscr{P}_2(\Rdrei))$ and \\ $v \in C^{1/2}([0,T],L^2_+(\Rdrei))$ such that the following holds for $k\to\infty$:
\begin{enumerate}[(a)]
\item For fixed $t\in[0,T]$, $u_{\tau_k}\to u$ narrowly in $\mathscr{P}(\Rdrei)$,
\item $v_{\tau_k}\to v$ uniformly with respect to $t\in[0,T]$ in $L^2(\Rdrei)$,
\item $u_{\tau_k}\rightharpoonup u$ weakly in $L^2([0,T],W^{1,2}(\Rdrei))$,
\item $v_{\tau_k}\rightharpoonup v$ weakly in $L^2([0,T],W^{2,2}(\Rdrei))$,
\item $u_{\tau_k}\to u$ strongly in $L^2([0,T],L^2(\Omega))$ for all bounded domains $\Omega\subset\Rdrei$,
\item $v_{\tau_k}\to v$ strongly in $L^2([0,T],W^{1,2}(\Omega))$ for all bounded domains $\Omega\subset\Rdrei$.
\end{enumerate}
\end{prop}

Now, to complete the proof of Theorem \ref{thm:existence}, one needs to verify that the obtained limit curve $(u,v)$ indeed satisfies the weak formulation \eqref{eq:pde_weak1}\&\eqref{eq:pde_weak2}. This will be omitted here for the sake of brevity.

\section{The stationary solution}\label{sec:stat}

In this section, we provide the characterization of a stationary state of system \eqref{eq:pde_u}\&\eqref{eq:pde_v} and prove some relevant properties.

\subsection{Existence and uniqueness}\label{subsec:ex+un}

At first, we show existence and uniqueness of the stationary solution to system \eqref{eq:pde_u}\&\eqref{eq:pde_v}:

\begin{prop}\label{prop:ex_infty}
For each sufficiently small $\eps>0$, there exists a unique minimizer $(u_\infty,v_\infty)\in \theX\cap(W^{1,2}(\Rdrei)\times W^{2,2}(\Rdrei))$ of $\calH$, for which the following holds:\\
$(u_\infty,v_\infty)$ is a stationary solution to \eqref{eq:pde_u}\&\eqref{eq:pde_v} and to the following Euler-Lagrange system
\begin{align}
&\Delta v_\infty-\kappa v_\infty=\eps u_\infty\phi'(v_\infty),\label{eq:v_infty}\\
&u_\infty=[U_\eps-W-\eps\phi(v_\infty)]_+,\label{eq:u_infty}
\end{align}
where $U_\eps\in\R$ is chosen such that $\|u_\infty\|_{L^1(\Rdrei)}=1$, and $[\cdot]_+$ denotes the positive part.\\
Moreover, there is a $V>0$ such that $v_\infty\in C^0(\Rdrei)$ satisfies $\|v_\infty\|_{L^\infty(\Rdrei)}\le V$, independent of $\eps>0$.
\end{prop}

\begin{proof}
We prove that $\calH$ possesses a unique minimizer $(u_\infty,v_\infty)$. Let therefore be given a minimizing sequence $(u_n,v_n)_{n\in\N}$ such that $\lim_{n\to\infty}\calH(u_n,v_n)=\inf\calH >-\infty$. As the sequence $(\calH(u_n,v_n))_{n\in\N}$ is bounded, we can, by the same argument as in the proof of Proposition \ref{prop:propH}, extract a (non-relabelled) subsequence, on which $(u_n)_{n\in\N}$ converges weakly in $L^2(\Rdrei)$ to some $u_\infty\in L^2_+(\Rdrei)$ and $(v_n)_{n\in\N}$ converges weakly in $W^{1,2}(\Rdrei)$ to some $v_\infty\in L^2_+(\Rdrei)\cap W^{1,2}(\Rdrei)$ as $n\to\infty$. By the same argument as in the proof of Proposition \ref{prop:propH}(b), $(u_\infty,v_\infty)$ is indeed a minimizer of $\calH$ and hence an element of $\theX\cap(L^2(\Rdrei)\times W^{1,2}(\Rdrei))$.

Since $(u_\infty,v_\infty)\in\mathrm{argmin}\,\calH_\tau(\cdot\,|\,u_\infty,v_\infty)$ for arbitrary $\tau>0$, Proposition \ref{prop:reg_min} immediately yields
\begin{align*}
\|u_\infty\|_{W^{1,2}(\Rdrei)}^2+\|v_\infty\|_{W^{2,2}(\Rdrei)}^2&\le V_0(\calH(u_\infty,v_\infty)+\|\Delta W\|_{L^\infty(\Rdrei)}+V_1),
\end{align*}
for some constants $V_0$, $V_1>0$. Because of the continuous embedding of $W^{2,2}(\Rdrei)$ into $C^0(\Rdrei)$, it follows that $\|v_\infty\|_{L^\infty(\Rdrei)}\le V$ for some $V>0$.

Uniqueness of the minimizer is, by \cite[Thm. 5.32]{villani2003}, a consequence of $\lambda'$-geodesic convexity of $\calH$ with respect to the distance $\dist'((u_1,v_1),(u_2,v_2)):=\sqrt{\|u_1-u_2\|_{L^2(\Rdrei)}^2+\|v_1-v_2\|_{L^2(\Rdrei)}^2}$ for some $\lambda'>0$ as proved in Proposition \ref{prop:propH}(c).

It remains to show that there is a set of Euler-Lagrange equations characterizing $(u_\infty,v_\infty)$. 

The following variational inequality holds:
\begin{align}
\int_\Rdrei(u_\infty+W+\eps\phi(v_\infty))\tilde u \dd x+\int_\Rdrei(-\Delta v_\infty+\kappa v_\infty+\eps u_\infty\phi'(v_\infty))\tilde v \dd x&\ge0,\label{eq:elineq}
\end{align}
for arbitrary maps $\tilde u$, $\tilde v$ such that $u_\infty+\tilde u\ge 0$ on $\Rdrei$ and $\int_\Rdrei \tilde u\dd x=0$.

First, we consider the second component and thus set $\tilde u=0$ in \eqref{eq:elineq}. As there are no constraints on $v_\infty$, it is allowed to replace $\tilde v$ by $-\tilde v$ in \eqref{eq:elineq}, yielding equality and hence, \eqref{eq:v_infty}.

Second, we consider the first component and set $\tilde v=0$ in \eqref{eq:elineq}. For arbitrary $\psi$ such that $\int_\Rdrei \psi\dd x\le 1$ and $\psi+u_\infty\ge 0$ on $\Rdrei$, we put
\begin{align*}
\tilde u_\psi&:=\frac12 \psi-\frac12 u_\infty\int_\Rdrei\psi\dd x,
\end{align*}
and observe that $u_\infty+u_\psi\ge 0$ on $\Rdrei$ and $\int_\Rdrei \tilde u_\psi\dd x=0$, since $u_\infty$ has mass equal to 1. By straightforward calculation, we obtain
\begin{align}
\label{eq:elineq2}
0&\le \int_\Rdrei (u_\infty+W+\eps\phi(v_\infty)-U_\eps)\psi\dd x,
\end{align}
for all $\psi$ as above and the constant
\begin{align*}
U_\eps:=\int_\Rdrei (u_\infty^2+Wu_\infty+\eps u_\infty \phi(v_\infty))\dd x\in\R.
\end{align*}

Fix $x\in\Rdrei$. If $u_\infty(x)>0$, choosing $\psi$ supported on a small neighbourhood of $x$ and replacing by $-\psi$ in \eqref{eq:elineq2} eventually yields
\begin{align*}
u_\infty(x)&= U_\eps-W(x)-\eps\phi(v_\infty(x)).
\end{align*}
If $u_\infty(x)=0$, we obtain
\begin{align*}
U_\eps-W(x)-\eps\phi(v_\infty(x))&\le 0,
\end{align*}
Hence, for all $x\in\Rdrei$,
\begin{align*}
u_\infty(x)&=[U_\eps-W(x)-\eps\phi(v_\infty(x))]_+.\qedhere
\end{align*}
\end{proof}

\subsection{Properties}\label{subsec:est}

As a preparation to prove some crucial regularity estimates on the stationary solution $(u_\infty,v_\infty)$, several properties of solutions to the elliptic partial differential equation $-\Delta h+\kappa h=f$ are needed.

Therefore, we introduce for $\kappa>0$ the \emph{Yukawa potential} (also called \emph{screened Coulomb} or \emph{Bessel potential}) $\krnl_\kappa$ by
\begin{align}
\krnl_\kappa(x)&:=\frac{1}{4\pi|x|}\exp(-\sqrt{\kappa}|x|)\quad\forall x\in\Rdrei\backslash\{0\}.\label{eq:yukawa}
\end{align}
Additionally, we define for $\sigma>0$ the kernel $\yucca_\sigma$ by
\begin{align*}
\yucca_\sigma&:=\frac{1}{\sigma}\krnl_{\frac{1}{\sigma}}.
\end{align*}
In subsequent parts of this work, we will need the \emph{iterates} $\yucca_\sigma^k$ for $k\in\N$ defined inductively by
\begin{align*}
 \yucca_\sigma^1 := \yucca_\sigma, \quad \yucca_\sigma^{k+1} := \yucca_\sigma *\yucca_\sigma^k.
\end{align*}
The relevant properties of $\krnl_\kappa$ and $\yucca_\sigma$ are summarized in Lemma \ref{lemma:reg_yuk} below. For the proof, we refer to Appendix \ref{app:yukawa}.

\begin{lemma}[Yukawa potential]\label{lemma:reg_yuk}
The following statements hold for all $\kappa>0$, $\sigma>0$ and $k\in\N$:
\begin{enumerate}[(a)]
\item $\krnl_\kappa$ and $\yucca_\sigma$ are the fundamental solutions to $-\Delta h+\kappa h=f$ and $-\sigma\Delta h+h=f$ on $\Rdrei$, respectively.
\item Let $p>1$. If $f\in L^p(\Rdrei)$, then $\krnl_\kappa*f\in W^{2,p}(\Rdrei)$ and
\begin{align}
\kappa\|\krnl_\kappa*f\|_{L^p(\Rdrei)}+\sqrt{\kappa}\|\dff (\krnl_\kappa*f)\|_{L^p(\Rdrei)}+\|\dff ^2(\krnl_\kappa*f)\|_{L^p(\Rdrei)}&\le C_p\|f\|_{L^p(\Rdrei)},\label{eq:reg_yuk}
\end{align}
for some $p$-dependent constant $C_p>0$. (Note that this fact is not obvious as $\dff ^2(\krnl_\kappa)\notin L^1(\Rdrei,\R^{3\times 3})$.)
\item For all $x\in\Rdrei\backslash\{0\}$,
\begin{align*}
 \yucca_\sigma(x) = \int_0^\infty \heat_{\sigma t}(x)e^{-t}\dd t,
\end{align*}
where $\heat_{t}$ is the \emph{heat kernel} on $\Rdrei$ at time $t>0$, i.e.
\begin{align*}
 \heat_t(\xi) = t^{-3/2}\heat_1(t^{-1/2}\xi), \quad\text{with}\quad
 \heat_1(\zeta) = (4\pi)^{-3/2}\exp\Big(-\frac14|\zeta|^2\Big).
\end{align*}
Additionally, one has
\begin{align}
   \label{eq:yuccaiter}
   \yucca_\sigma^k = \int_0^\infty \heat_{\sigma r}\frac{r^{k-1}e^{-r}}{\Gamma(k)}\dd r.
 \end{align}
 Moreover, $\yucca_\sigma^k\in W^{1,q}(\Rdrei)$ for each $q\in[1,3/2)$,
 and there are universal constants $Y_q$ such that
 \begin{align}
   \label{eq:yuccaest}
   \|\dff\yucca_\sigma^k\|_{L^q(\Rdrei)} \le Y_q (\sigma k)^{-Q}
   \quad \text{where} \quad Q:=2-\frac3{2q}\in[1/2,1).
 \end{align}
\end{enumerate}
\end{lemma}

Now, we are in position to prove several estimates on the stationary solution:
\begin{prop}[Estimates on the stationary solution]\label{prop:est_infty}
The following uniform estimates hold for all $x\in\Rdrei$:
\begin{enumerate}[(a)]
\item $u_\infty(x)\le U_0-\eps V \phi'(0)$, where $U_0\in\R$ is chosen in such a way that $\int_\Rdrei[U_0-W]_+\dd x=1$ and $V>0$ is the constant from Prop. \ref{prop:ex_infty}.
\item $|\dff v_\infty(x)|\le C\eps$ for some constant $C>0$.
\item $-C'\eps \eins\le\dff ^2v_\infty(x)\le C'\eps \eins$ in the sense of symmetric matrices, for some constant $C'>0$.
\end{enumerate}
\end{prop}

\begin{proof}
\begin{enumerate}[(a)]
\item We first prove that $U_\eps\le U_0+\eps \phi(0)$, which in turn follows if
\begin{align*}
\int_\Rdrei [U_0+\eps \phi(0)-W-\eps \phi(v_\infty)]_+\dd x&\ge 1.
\end{align*}
One has
\begin{align}
&\int_\Rdrei [U_0+\eps \phi(0)-W-\eps \phi(v_\infty)]_+\dd x\nonumber\\
&=\int_{\{U_0-W\ge 0\}}[U_0-W+\eps(\phi(0)-\phi(v_\infty))]\dd x+\int_{\{0>U_0-W\ge \eps(\phi(v_\infty)-\phi(0))\}}[U_0-W+\eps(\phi(0)-\phi(v_\infty))]\dd x.\label{eq:lemmaL2}
\end{align}
From $\phi(0)-\phi(v_\infty)\ge 0$ and the definition of $U_0$, we deduce that the first term in \eqref{eq:lemmaL2} is larger or equal to $1$. The second term in \eqref{eq:lemmaL2} is nonnegative because the integrand is nonnegative on the domain of integration. 

Now, if $u_\infty(x)>0$ for some $x\in\Rdrei$, we also have due to convexity of $\phi$:
\begin{align*}
u_\infty(x)&\le U_\eps-W(x)-\eps\phi(0)-\eps v_\infty(x)\phi'(0)\le U_0+\eps\phi(0)-\eps\phi(0)-\eps V\phi'(0),
\end{align*}
from which the desired estimate follws.
\item Define
\begin{align*}
f_v:\Rdrei\to\R,\quad f(x):=-\eps [U_\eps-W(x)-\eps\phi(v(x))]_+\phi'(v(x)).
\end{align*}
Then, $f_v\in L^\infty(\Rdrei)$ with compact support $\mathrm{supp}(f_v)\subset B_R(0)$ where $R>0$ can be chosen independent of $\eps\in(0,1)$. Moreover, by Lemma \ref{lemma:reg_yuk}(a), $(u_\infty,v_\infty)$ is the solution to the integral equation
\begin{align*}
v&=-(\krnl_\kappa*f_v),
\end{align*}
with the Yukawa potential $\krnl_\kappa$ defined in \eqref{eq:yukawa}. Since $W^{2,4}(\Rdrei)$ is continuously embedded in $C^1(\Rdrei)$ \cite[Appendix, sec. (45) et seq.]{zeidler1990} and $f_v\in L^4(\Rdrei)$, we deduce from Lemma \ref{lemma:reg_yuk}(b) that
\begin{align*}
\|v\|_{C^1(\Rdrei)}&\le \tilde C\|f_v\|_{L^4(\Rdrei)},
\end{align*}
for some constant $\tilde C>0$. Hence, we obtain (b) by using (a):
\begin{align*}
\|\dff v_\infty\|_{L^\infty(\Rdrei)}&\le \tilde C\|f_{v_\infty}\|_{L^4(\Rdrei)}\le \tilde C\eps(U_0-\eps V \phi'(0))|\phi'(0)||B_R(0)|^{1/4}=:C\eps.
\end{align*}
\item First, consider $x\in \Rdrei\backslash B_{R+1}(0)$, where $R>0$ is such that $\mathrm{supp}(f_{v_\infty})\subset B_R(0)$. Smoothness of $\krnl_\kappa$ on $\Rdrei\backslash\{0\}$ yields for all $i,j\in\{1,2,3\}$:
\begin{align*}
|\partial_i\partial_j v_\infty(x)|&=\left|\int_{B_R(0)}(\partial_{x_i}\partial_{x_j}\krnl_\kappa (x-y))f_{v_\infty}(y)\dd y\right|\\
&=\left|\int_{x+B_R(0)}(\partial_{i}\partial_{j}\krnl_\kappa (z))f_{v_\infty}(x-z)\dd z\right|,
\end{align*}
where the last equality follows by the transformation $z:=x-y$. Obviously, we obtain the estimate
\begin{align*}
|\partial_i\partial_j v_\infty(x)|&\le \|f_{v_\infty}\|_{L^\infty(\Rdrei)}\int_{\Rdrei\backslash B_1(0)}|\partial_i\partial_j \krnl_\kappa(z)|\dd z.
\end{align*}
Since for $|z|\ge 1$, one has (see Appendix \ref{app:holder} for the derivatives of $\krnl_\kappa$)
\begin{align*}
|\partial_i\partial_j \krnl_\kappa(z)|&\le \frac{C(\kappa)\exp(-\sqrt{\kappa}z)}{4\pi|z|},
\end{align*}
we arrive at
\begin{align*}
|\partial_i\partial_j v_\infty(x)|&\le C(\kappa)\|f_{v_\infty}\|_{L^\infty(\Rdrei)}\int_1^\infty \exp(-\sqrt{\kappa}r)r\dd r,
\end{align*}
the last integral obviously being finite.

Consider now the case $|x|\le R+1$ and set $y:=(R+2)e_1\neq x$. By the triangular inequality, we have for $\alpha\in (0,1)$ that
\begin{align*}
|\partial_i\partial_j v_\infty(x)|\le |\partial_i\partial_jv_\infty (y)|+\frac{|\partial_i\partial_j v_\infty(x)-\partial_i\partial_j v_\infty(y)|}{|x-y|^\alpha}{|x-y|^\alpha}.
\end{align*}
By the arguments above, $f_{v_\infty}$ is $\alpha$-H\"older-continuous for some $\alpha\in(0,1)$ since $u_\infty$ is Lipschitz-continuous and of compact support. By Lemma \ref{lemma:holder} in Appendix \ref{app:holder}, we know that there exists $C>0$ such that
\begin{align*}
[\partial_i\partial_j v_\infty]_{C^{0,\alpha}(\Rdrei)}&\le C[f_{v_\infty}]_{C^{0,\alpha}(\Rdrei)}.
\end{align*}
Hence, since $|x-y|\le 2R+3$, one has
\begin{align*}
|\partial_i\partial_j v_\infty(x)|&\le |\partial_i\partial_jv_\infty (y)|+C(2R+3)^\alpha[f_{v_\infty}]_{C^{0,\alpha}(\Rdrei)}.
\end{align*}
Combining both cases yields 
\begin{align*}
|\partial_i\partial_j v_\infty(x)|&\le |\partial_i\partial_jv_\infty ((R+2)e_1)|+C(2R+3)^\alpha[f_{v_\infty}]_{C^{0,\alpha}(\Rdrei)}\\
&\le C_0\|f_{v_\infty}\|_{L^\infty(\Rdrei)}+C_1 \|f_{v_\infty}\|_{W^{1,\infty}(\Rdrei)},
\end{align*}
for some $C_0,C_1>0$ and \emph{all} $x\in\Rdrei$. Using (a) and (b), it is straightforward to conclude that there exists $C_2>0$ with
\begin{align*}
(\|\dff f\|_{C^0(\Rdrei)}+\|f\|_{L_\infty(\Rdrei)})&\le C_2 \eps.
\end{align*}
All in all, we proved the existence of $C_3>0$ such that for all $x\in\Rdrei$ and all $i,j\in\{1,2,3\}$:
\begin{align*}
|\partial_i\partial_j v_\infty(x)|&\le C_3\eps.
\end{align*}
Obviously, this estimate yields the assertion (for a different constant $C'>0$).
\end{enumerate}
\end{proof}

\section{Convergence to equilibrium}\label{sec:conv}
In this section, we prove Theorem \ref{thm:convergence}. The strategy of proof is as follows: We first show that the entropy $\calH(u,v)-\calH_\infty$ can indeed be decomposed as in \eqref{eq:decompose}. Furthermore, the second component $v_\tau^n$ of the discrete solution admits a control estimate enabling us to prove boundedness of the auxiliary entropy $\lyp_u(u)+\lyp_v(v)$ in \eqref{eq:decompose} for large times. From that, we can deduce an explicit temporal bound such that exponential decay to zero of this entropy occurs for sufficiently large times. These estimates are finally converted into the desired estimate for the continuous weak solution, completing the proof of Theorem \ref{thm:convergence}.\\

Since our claim only concerns the solutions $(u,v)$ to \eqref{eq:pde_u}--\eqref{eq:init}
that are constructed as in the proof of Theorem \ref{thm:existence}, i.e. by the minimizing movement scheme,
we assume in the following that we are given a family of time-discrete approximations $(u_\tau^n,v_\tau^n)_{n\in\N}$
that converge to the weak solution $(u,v)$ in the sense discussed in Section \ref{sec:existence} as $\tau\downarrow0$. Therefore, we may assume without loss of generality that $\tau>0$ is sufficiently small.

Throughout this section, we shall use the abbreviation $[a]_\tau := \frac1\tau\log (1+a\tau)$, where $a>0$.
Note that if, for every $\tau>0$ and an index $m_\tau\in\N$ given such that $m_\tau\tau\ge T$ with a fixed $T\ge0$,
then
\begin{align}
 \label{eq:tauconverge}
 (1+a\tau)^{-m_\tau}\le e^{-[a]_\tau T}\downarrow e^{-a T} \quad\text{as $\tau\downarrow0$.}
\end{align}
In order to keep track of the dependencies of certain quantities on $\eps$,
we are going to define several positive numbers $\eps_j$ such that the estimates in a certain proof are uniform with respect to $\eps\in(0,\eps_j)$.
When we want to emphasize that a quantity is independent of $\eps\in(0,\eps_j)$
-- and also of $\tau$ and the initial condition $(u_0,v_0)$ -- we call it a \emph{system constant}.
System constants are (in principle) expressible as a function of $\lambda_0$, $\kappa$, $\phi$ and truely universal constants. Finally, we write $\calH_\infty:=\calH(u_\infty,v_\infty)$.

\subsection{Decomposition of the entropy}
The key element in the proof of Theorem \ref{thm:convergence} is the decomposition of the entropy functional as announced in \eqref{eq:decompose}.
Introduce the \emph{perturbed potential} $W_\eps$ by
\begin{align}
 \label{eq:perturbed_pot}
 W_{\eps}(x):=W(x)+\eps\phi(v_\infty(x)).
\end{align}
Recall that $(u_\infty,v_\infty)$ is the minimizer of $\calH$ on $\theX$,
and define
\begin{align*}
 \lyp_u(u) &:= \int_\Rdrei \Big(\frac12(u^2-u_\infty^2) + W_\eps(u-u_\infty)\Big)\dd x, \\
 \lyp_v(v) &:= \int_\Rdrei \frac12\big(|\dff(v-v_\infty)|^2 + \kappa(v-v_\infty)^2\big)\dd x, \\
 \lyp_*(u,v) &:= \int_\Rdrei \big(u[\phi(v)-\phi(v_\infty)]-u_\infty\phi'(v_\infty)[v-v_\infty]\big)\dd x.
\end{align*}
Finally, let $\lyp(u,v):=\lyp_u(u)+\lyp_v(v)$ denote the \emph{auxiliary entropy}.
\begin{lemma}
 The decomposition \eqref{eq:decompose} holds:
\begin{align*}
 \calH(u,v)-\calH_\infty  = \lyp(u,v) + \eps\lyp_*(u,v).
\end{align*}
\end{lemma}
\begin{proof}
 By the properties of $\phi$ and the fact that $u_\infty$ has compact support, $\lyp_*$ is well-defined on all of $\theX$,
 while $\lyp_u$ and $\lyp_v$ are finite precisely on $(\prb\cap L^2)(\Rdrei)$ and $W^{1,2}(\Rdrei)$, respectively.
 Thus, both sides in \eqref{eq:decompose} are finite on the same subset of $\theX$.
 Now, for every such pair $(u,v)$, we have on the one hand that
 \begin{align}
   \label{eq:lypu}
   \lyp_u(u) = \int_\Rdrei \Big(\frac12u^2 +  uW + \eps u\phi(v_\infty)\Big)\dd x - \int_\Rdrei \Big(\frac12u_\infty^2 + u_\infty W + \eps u_\infty\phi(v_\infty)\Big)\dd x ,
 \end{align}
 and on the other hand that
 \begin{align*}
   \lyp_v(v) &= \int_\Rdrei \Big(\frac12|\dff v|^2 + \frac\kappa2v^2\Big)\dd x
   + \int_\Rdrei \Big(\frac12|\dff v_\infty|^2 + \frac\kappa2v_\infty^2\Big)\dd x
   - \int_\Rdrei (\dff v\cdot\dff v_\infty + \kappa vv_\infty)\dd x.
 \end{align*}
 Integration by parts in the last integral yields, recalling the defining equation \eqref{eq:v_infty} for $v_\infty$, that
 \begin{align*}
   -\int_\Rdrei (\dff v\cdot\dff v_\infty + \kappa vv_\infty)\dd x
   = \int_\Rdrei (\Delta v_\infty - \kappa v_\infty)v\dd x
   = \eps \int_\Rdrei u_\infty\phi'(v_\infty)v\dd x.
 \end{align*}
 Similarly, integration by parts in the middle integral leads to
 \begin{align*}
   \int_\Rdrei \Big(\frac12|\dff v_\infty|^2 + \frac\kappa2v_\infty^2\Big)\dd x
   = -\frac\eps2 \int_\Rdrei u_\infty\phi'(v_\infty)v_\infty\dd x.
 \end{align*}
 And so,
 \begin{align}
   \label{eq:lypv}
   \begin{split}
     \lyp_v(v) &= \int_\Rdrei \Big(\frac12|\dff v|^2 + \frac\kappa2v^2\Big)\dd x
     - \int_\Rdrei \Big(\frac12|\dff v_\infty|^2 + \frac\kappa2v_\infty^2\Big)\dd x  \\
     & \qquad +\eps \int_\Rdrei u_\infty\phi'(v_\infty)(v-v_\infty)\dd x.
   \end{split}
 \end{align}
 Combining \eqref{eq:lypu} and \eqref{eq:lypv} with the definition of $\lyp_*$ yields \eqref{eq:decompose}.
\end{proof}

We summarize some useful properties of the auxiliary entropy $\lyp$ in the following.
\begin{prop}[Properties of $\lyp$]
 \label{prop:prop_L}
 There are constants $K,L>0$ and some $\eps_0>0$ such that the following is true for every $\eps\in(0,\eps_0)$:
 \begin{enumerate}[(a)]
 \item $W_{\eps}\in C^2(\Rdrei)$ is $\lambda_\eps$-convex with $\lambda_\eps:=\lambda_0-L\eps>0$.
 \item $\lyp_u$ is $\lambda_\eps$-convex in $(\prb(\Rdrei),\W_2)$,
   and for every $u\in(\prb\cap W^{1,2})(\Rdrei)$, one has
   \begin{align}
     \label{eq:subdiff_u}
     \frac12\|u-u_\infty\|_{L^2(\Rdrei)}^2
     \le\lyp_u(u)
     \le\frac1{2\lambda_\eps}\int_\Rdrei u|\dff(u+W_{\eps})|^2\dd x.
   \end{align}
 \item $\lyp_v$ is $\kappa$-convex in $L^2(\Rdrei)$,
   and for every $v\in W^{2,2}(\Rdrei)$, one has
   \begin{align}
     \label{eq:subdiff_v}
     \frac\kappa2\|v-v_\infty\|_{L^2(\Rdrei)}^2
     \le \lyp_v(v)
\le\frac1{2\kappa}\int_\Rdrei\big(\Delta(v-v_\infty)-\kappa(v-v_\infty)\big)^2\dd x.
   \end{align}
 \item For every $(u,v)\in\theX$,
   \begin{align}
     \label{eq:lypbyH}
     \lyp(u,v) \le (1+K\eps)\big(\calH(u,v)-\calH_\infty\big).
   \end{align}
 \end{enumerate}
\end{prop}
\begin{proof}
 \begin{enumerate}[(a)]
\item  Since $W_\eps=W+\eps\phi(v_\infty)$, the chain rule yields
 \begin{align*}
   \dff^2W_\eps = \dff^2W + \eps\phi''(v_\infty)\dff v_\infty\otimes\dff v_\infty + \eps\phi'(v_\infty)\dff^2v_\infty.
 \end{align*}
 Using our assumptions on $\phi$ and by Proposition \ref{prop:est_infty},
 there are some $L>0$ and some $\eps_0$ such that
 \begin{align*}
   \phi''(v_\infty)\dff v_\infty\otimes\dff v_\infty + \phi'(v_\infty)\dff^2v_\infty \ge -L\eins
 \end{align*}
 holds uniformly with respect to $\eps\in(0,\eps_0)$.
 And thus also $\dff^2W_\eps\ge\lambda_\eps\eins$, with the indicated definition of $\lambda_\eps$.
 \smallskip

 \item  Since $W_\eps$ is $\lambda_\eps$-convex,
 also $\lyp_u$ is $\lambda_\eps$-geodesically convex in $\W_2$ because it is
 the sum of a $0$-geodesically convex functional and a $\lambda_\eps$-geodesically convex functional,
 see Theorem \ref{thm:crit_conv}.

 The Wasserstein subdifferential of $\lyp_u$ has been calculated in \cite[Lemma 10.4.1]{savare2008}. Together with \eqref{eq:subdiff}, this shows the second inequality in \eqref{eq:subdiff_u}.
 Concerning the first inequality, observe that
 \begin{align*}
   \lyp_u(u) = \frac12\int_\Rdrei(u-u_\infty)^2\dd x+\int_\Rdrei (W_\eps+u_\infty)(u-u_\infty)\dd x.
 \end{align*}
 It thus suffices to prove nonnegativity of the second integral term for all $u\in\prb(\Rdrei)$. First, as $u$ and $u_\infty$ have equal mass, and by the definition of $u_\infty$,
\begin{align*}
0&=\int_\Rdrei(u_\infty-u)\dd x=\int_{\{U_\eps-W_\eps> 0\}}u_\infty \dd x-\int_\Rdrei u \dd x,
\end{align*}
and consequently
\begin{align}
\int_{\{U_\eps-W_\eps> 0\}}(u-u_\infty)\dd x&=-\int_{\{U_\eps-W_\eps\le 0\}}u\dd x.\label{eq:equalmass}
\end{align}
Also, by definition of $u_\infty$,
\begin{align*}
\int_\Rdrei (W_\eps+u_\infty)(u-u_\infty)\dd x&=\int_{\{U_\eps-W_\eps> 0\}}U_\eps(u-u_\infty)\dd x+\int_{\{U_\eps-W_\eps\le 0\}}W_\eps u\dd x.
\end{align*}
Combining this with \eqref{eq:equalmass} yields
\begin{align*}
\int_{\{U_\eps-W_\eps> 0\}}U_\eps(u-u_\infty)\dd x+\int_{\{U_\eps-W_\eps\le 0\}}W_\eps u\dd x&=\int_{\{U_\eps-W_\eps\le 0\}}(W_\eps-U_\eps) u\dd x\ge 0,
\end{align*}
as the integrand is nonnegative on the domain of integration.

 \item  This is an immediate consequence of \eqref{eq:subdiff} for the $L^2$ subdifferential of $\lyp_v$.

 \item  Since $\phi$ is convex, we have
 \begin{align*}
   \phi(v)-\phi(v_\infty)-\phi'(v_\infty)[v-v_\infty]\ge0,
 \end{align*}
 and so we can estimate $\lyp_*$ from below as follows:
 \begin{align*}
   \lyp_*(u,v)
   &= \int_\Rdrei (u-u_\infty)[\phi(v)-\phi(v_\infty)]\dd x + \int_\Rdrei \big(\phi(v)-\phi(v_\infty)-\phi'(v_\infty)[v-v_\infty]\big)\dd x \\
   &\ge -\frac12\int_\Rdrei (u-u_\infty)^2\dd x - \frac{\phi'(0)^2}2\int_\Rdrei (v-v_\infty)^2\dd x \\
   &\ge -\lyp_u(u) - \frac{\phi'(0)^2}{\kappa}\lyp_v(v),
 \end{align*}
 using the properties (b) and (c) above.
 By \eqref{eq:decompose}, we conclude
 \begin{align*}
   (1-K'\eps)\lyp(u,v) = \calH(u,v)-\calH_\infty \quad \text{with} \quad K':=\max\Big(1,\frac{\phi'(0)^2}{\kappa}\Big),
 \end{align*}
 which clearly implies \eqref{eq:lypbyH} for all $\eps\in(0,\eps_0)$, possibly after diminishing $\eps_0$.
\end{enumerate}
\end{proof}

\subsection{Dissipation}
We can now formulate the main \textit{a priori} estimate for the time-discrete solution.
\begin{prop}
 Given $(\tilde u,\tilde v)\in\theX$ with $\calH(\tilde u,\tilde v)<\infty$, let $(u,v)\in\theX$ be a minimizer
 of the functional $\calH_\tau(\,\cdot\,|\tilde u,\tilde v)$ introduced in \eqref{eq:yoshida}.
 Then
 \begin{align}
   \label{eq:1}
   \lyp_u(u) + \tau\dsp_u(u,v) \le \lyp_u(\tilde u)
   \quad\text{and}\quad
   \lyp_v(v) + \tau\dsp_v(u,v) \le \lyp_v(\tilde v),
 \end{align}
 where the dissipation terms are given by
 \begin{align}
   \label{eq:du}
   \dsp_u(u,v) &= \Big(1-\frac\eps2\Big)\int_\Rdrei u|\dff(u+W_\eps)|^2\dd x
   - \frac\eps2 \int_\Rdrei u\big|\dff\big(\phi(v)-\phi(v_\infty)\big)\big|^2\dd x, \\
   \label{eq:dv}
   \dsp_v(u,v) &= \Big(1-\frac\eps2\Big)\int_\Rdrei \big(\Delta(v-v_\infty)-\kappa(v-v_\infty)\big)^2\dd x
   - \frac\eps2 \int_\Rdrei \big(u\phi'(v)-u_\infty\phi'(v_\infty)\big)^2\dd x.
 \end{align}
\end{prop}
\begin{proof}
 Naturally, these estimates are derived by means of the flow interchange lemma \ref{lem:flowinterchange}.

 For given $\nu>0$, introduce the regularized functional $\lyp_u^\nu=\lyp_u+\nu\ent$,
 where
 \begin{align*}
   \ent(u) = \int_\Rdrei u\log u\dd x.
 \end{align*}
 Note that $\ent$ is finite on $(\prb\cap L^2)(\Rdrei)$, see e.g.\ \cite[Lemma 5.3]{zinsl2012}.
 Moreover, $\lyp_u^\nu$ is $\lambda_\eps$-convex in $\W_2$ by Theorem \ref{thm:crit_conv}.
 We claim that the $\lambda_\eps$-flow associated to $\lyp_u^\nu$ satisfies the evolution equation
 \begin{align}
   \label{eq:nuflow}
   \partial_s\aU = \nu\Delta\aU + \frac12\Delta\aU^2 + \dv(\aU\dff W_\eps).
 \end{align}
 Since $\nu>0$, this equation is strictly parabolic.
 Therefore, for every initial condition $\aU_0\in(\prb\cap L^2)(\Rdrei)$,
 there exists a smooth and positive solution $\aU:\R_+\times\Rdrei\to\R$ such that $\aU(s,\cdot)\to\aU_0$ both in $\W_2$ and in $L^2(\Rdrei)$ as $s\downarrow0$.
 By \cite[Theorem 11.2.8]{savare2008}, the solution operator to \eqref{eq:nuflow} can be identified with the $\lambda_\eps$-flow of $\lyp_u^\nu$.

 Now, let $\aU$ be the smooth solution to \eqref{eq:nuflow} with initial condition $\aU_0=u$.
 By smoothness of $\aU$, the equation \eqref{eq:nuflow} is satisfied in the classical sense at every time $s>0$,
 and the following integration by parts is justified:
 \begin{align*}
   -\frac{\dn}{\dn s}\calH(\aU,v)
   &= -\int_\Rdrei \big[\aU+W_\eps+\eps(\phi(v)-\phi(v_\infty))\big]\dv\big[\aU\dff(\aU+W_\eps)+\nu\dff\aU\big]\dd x \\
   &= \int_\Rdrei \aU|\dff(\aU+W_\eps)|^2\dd x + \eps\int_\Rdrei\aU\dff(\phi(v)-\phi(v_\infty))\cdot\dff(\aU+W_\eps)\dd x\\
   & \qquad + \nu \int_\Rdrei \dff\big[\aU+W+\eps\phi(v)\big]\cdot\dff\aU\dd x.
 \end{align*}
 The very last integral has already been estimated in the proof of Proposition \ref{prop:reg_min}.
 Rewriting the middle integral by means of the elementary inequality
 \begin{align}
   \label{eq:binom}
   2ab \le  a^2 + b^2,
 \end{align}
 we arrive at
 \begin{align*}
   -\frac{\dn}{\dn s}\calH(\aU,v)
   &\ge \Big(1-\frac{\eps}{2}\Big)\int_\Rdrei \aU|\dff(\aU+W_\eps)|^2\dd x - \frac\eps2 \int_\Rdrei \aU\big|\dff\big(\phi(v)-\phi(v_\infty)\big)\big|^2\dd x \\
   & \qquad - \nu K\big(\|\aU\|_{L^2(\Rdrei)}^2 + \|v\|_{W^{1,2}(\Rdrei)}^2\big).
 \end{align*}
 We pass to the limit $s\downarrow0$.
 Recall that $\aU$ converges (strongly) to its initial datum $\aU_0=u$ in $L^2(\Rdrei)$,
 and observe that the expressions on the right-hand side are lower semicontinuous with respect to that convergence.
 In fact, this is clear except perhaps for the first integral, which however can be rewritten, using integration by parts,
 in the form
 \begin{align*}
   \int_\Rdrei \aU|\dff(\aU+W_\eps)|^2\dd x
   = \frac49 \int_\Rdrei |\dff\aU^{3/2}|^2\dd x - \int_\Rdrei \aU^2\Delta W_\eps\dd x + \int_\Rdrei \aU|\nabla W_\eps|^2\dd x,
 \end{align*}
 in which the lower semicontinuity is obvious since $\Delta W_\eps\in L^\infty(\Rdrei)$.
 Applying now Lemma \ref{lem:flowinterchange},
 we arrive at
 \begin{align*}
   \lyp_u^\nu(u) + (1-\eps)\int_\Rdrei u|\dff(u+W_\eps)|^2\dd x
   & - \frac\eps2 \int_\Rdrei u\big|\dff\big(\phi(v)-\phi(v_\infty)\big)\big|^2\dd x \\
   & \qquad \le \lyp_u^\nu(\tilde u) + \nu K\big(\|u\|_{L^2(\Rdrei)}^2 + \|u\|_{W^{1,2}(\Rdrei)}^2\big).
 \end{align*}
 Finally, passage to the limit $\nu\downarrow0$ yields the dissipation \eqref{eq:du}.

The dissipation \eqref{eq:dv} is easier to obtain.
 It is immediate that the $\kappa$-flow in $L^2(\Rdrei)$ of $\lyp_v$ satisfies the linear parabolic evolution equation
 \begin{align}
   \label{eq:xflow}
   \partial_s\aV = \Delta(\aV-v_\infty) - \kappa(\aV-v_\infty).
 \end{align}
 Solutions $\aV$ to \eqref{eq:xflow} exist for arbitrary initial conditions $\aV_0\in L^2(\Rdrei)$,
 and they have at least the spatial regularity of $v_\infty$.
 Hence, with $\aV_0:=v$, we have, also recalling the defining equation \eqref{eq:v_infty} for $v_\infty$,
 \begin{align*}
   -\frac{\dn}{\dn s} \calH(u,\aV)
   & = \int_\Rdrei \big[\Delta(\aV-v_\infty) - \kappa(\aV-v_\infty) - \eps(u\phi'(\aV)-u_\infty\phi'(v_\infty))\big]\\
   & \qquad \cdot\big[\Delta(\aV-v_\infty) - \kappa(\aV-v_\infty)\big] \dd x\\
\end{align*}
 Another application of the elementary inequality \eqref{eq:binom} yields
 \begin{align*}
   -\frac{\dn}{\dn s} \calH(u,\aV)
   \ge \Big(1-\frac\eps2\Big)\int_\Rdrei \big[\Delta(\aV-v_\infty) - \kappa(\aV-v_\infty)\big]^2 \dd x
   - \frac\eps2 \int_\Rdrei (u\phi'(\aV)-u_\infty\phi'(v_\infty))^2\dd x.
 \end{align*}
 We pass to the limit $s\downarrow0$, so that $\aV$ converges to $v$ in $L^2(\Rdrei)$.
 The first integral is obviously lower semicontinuous.
 Concerning the second integral, note that the integrand converges pointwise a.e.\ on $\Rdrei$ on a subsequence,
 and that it is pointwise a.e.\ bounded by the integrable function $2\phi'(0)^2(u^2+u_\infty^2)$.
 Hence, we can pass to the limit using the dominated convergence theorem.
 Now another application of Lemma \ref{lem:flowinterchange} yields the desired result.
\end{proof}

We will need below two further estimates for the dissipation terms from \eqref{eq:du}\&\eqref{eq:dv}.
\begin{lemma}
 There is a constant $\theta>0$ such that for every $\eps\in(0,\eps_0)$ and every $u\in(\prb\cap W^{1,2})(\Rdrei)$,
 the following holds:
 \begin{align}
   \label{eq:u3}
   \|u\|_{L^3(\Rdrei)}^4 \le \theta\bigg(1+\int_\Rdrei u|\dff(u+W_\eps)|^2\dd x \bigg).
 \end{align}
\end{lemma}
\begin{proof}
 Integrating by parts, it is easily seen that
 \begin{align*}
   \frac49\int_\Rdrei \big|\dff u^{3/2}\big|^2\dd x + \int_\Rdrei u|\dff W_\eps|^2\dd x
   = \int_\Rdrei u|\dff(u+W_\eps)|^2\dd x + \int_\Rdrei u^2\Delta W_\eps\dd x.
 \end{align*}
 By Proposition \ref{prop:est_infty} on the regularity of $u_\infty$ and $v_\infty$,
 there exists a constant $C$ such that
 \begin{align*}
   \Delta W_\eps=\Delta W + \eps\phi'(v_\infty)\Delta v_\infty + \eps\phi''(v_\infty)|\dff v_\infty|^2 \le C \quad \text{on $\Rdrei$}
 \end{align*}
 for all $\eps\in(0,\eps_1)$.
 Moreover,
 \begin{align*}
   \frac12\int_\Rdrei u^2\dd x\le \int_\Rdrei u_\infty^2\dd x + \frac1{\lambda_\eps}\int_\Rdrei u|\dff(u+W_\eps)|^2\dd x
 \end{align*}
 by \eqref{eq:subdiff_u}.
 Invoking again Proposition \ref{prop:est_infty}, it follows that there exists an $\eps$-uniform constant $C'$
 such that
 \begin{align*}
   \|\dff u^{3/2}\|_{L^2(\Rdrei)}^2 \le C'\bigg(1+\int_\Rdrei u|\dff(u+W_\eps)|^2\dd x\bigg)
 \end{align*}
 holds for all $u\in\prb(\Rdrei)$.
 On the other hand, H\"older's and Sobolev's inequalities provide
 \begin{align*}
   \|u\|_{L^3(\Rdrei)} &\le \|u^{3/2}\|_{L^6(\Rdrei)}^{1/2}\|u\|_{L^1(\Rdrei)}^{1/4} \le C''\|\dff u^{3/2}\|_{L^2(\Rdrei)}^{1/2},
 \end{align*}
 where we have used that $u$ is of unit mass.
 Together, this yields \eqref{eq:u3}.
\end{proof}
\begin{lemma}
 For every $v\in W^{2,2}(\Rdrei)$,
 \begin{align}
   \label{eq:w22}
   \min(1,2\kappa,\kappa^2)\|v-v_\infty\|_{W^{2,2}(\Rdrei)}^2 \le \int_\Rdrei \big(\Delta(v-v_\infty)-\kappa(v-v_\infty)\big)^2\dd x.
 \end{align}
\end{lemma}
\begin{proof}
 Set $\bv:=v-v_\infty$ for brevity.
 Integration by parts yields
 \begin{align*}
   \int_\Rdrei (\Delta\bv-\kappa\bv)^2\dd x
   &= \int_\Rdrei (\Delta\bv)^2\dd x - 2\kappa \int_\Rdrei \bv\Delta\bv\dd x + \kappa^2\int_\Rdrei\bv^2\dd x \\
   &= \int_\Rdrei \|\dff^2\bv\|^2\dd x + 2\kappa\int_\Rdrei |\dff\bv|^2\dd x + \kappa^2\int_\Rdrei\bv^2\dd x,
 \end{align*}
 which clearly implies \eqref{eq:w22}.
\end{proof}

\subsection{Control on the $v$ component}
For our estimates below, we need some preliminaries concerning solutions to the time-discrete heat equation. Here, we use the iterates $\yucca_\sigma^k$ defined in \eqref{eq:yuccaiter} to write a semi-explicit representation of the components $v_\tau^n$ for a particular choice of $\sigma$.
\begin{lemma}
 For every $n\in\N$,
 \begin{align}
   \label{eq:vrep}
   v_\tau^n = (1+\kappa\tau)^{-n}\yucca_\sigma^n * v_0 + \tau\sum_{m=1}^n(1+\kappa\tau)^{-m}\yucca_\sigma^m * f_\tau^{n+1-m},
 \end{align}
 where we have set
 \begin{align*}
   f_\tau^k := -\eps u_\tau^k\phi'(v_\tau^k), \quad \sigma:=\frac\tau{1+\kappa\tau}.
 \end{align*}
\end{lemma}
\begin{proof}
 We proceed by induction on $n$.
 By the flow interchange lemma \ref{lem:flowinterchange}, using the auxiliary functional $\auxil(u,v):=\int_\Rdrei \gamma v \dd x$ for an arbitrary test function $\gamma\in C^\infty_c(\Rdrei)$, one sees by analogous (but easier) arguments as in the proof of \eqref{eq:dv} that $v_\tau^n$ is the -- unique in $L^2(\Rdrei)$ -- distributional solution
 to
 \begin{align*}
   v_\tau^n - \sigma \Delta v_\tau^n = (1+\kappa\tau)^{-1}v_\tau^{n-1} + \tau(1+\kappa\tau)^{-1}f_\tau^n.
 \end{align*}
 Hence it can be written as
 \begin{align*}
   v_\tau^n = (1+\kappa\tau)^{-1}\yucca_\sigma * v_\tau^{n-1} + \tau(1+\kappa\tau)^{-1}\yucca_\sigma * f_\tau^n.
 \end{align*}
 For $n=1$, this is \eqref{eq:vrep} because $v_\tau^0 = v_0$.
 Now, if $n>1$, and \eqref{eq:vrep} holds with $n-1$ in place of $n$,
 then
 \begin{align*}
   v_\tau^n & = (1+\kappa\tau)^{-n}\yucca_\sigma *(\yucca_\sigma^{n-1} * v_0) \\
   & \qquad + \tau\sum_{m=1}^{n-1}(1+\kappa\tau)^{-(m+1)}\yucca_\sigma *(\yucca_\sigma^m * f_\tau^{n-m})
   + \tau(1+\kappa\tau)^{-1}\yucca_\sigma * f_\tau^n.
 \end{align*}
 Using that $\yucca_\sigma *(\yucca_\sigma^k * f)=\yucca_\sigma^{k+1} * f$, we obtain \eqref{eq:vrep}.
\end{proof}
We are now able to prove the main result of this section.
\begin{prop}
 Provided that $v_0\in L^{6/5}(\Rdrei)$,
 then $\dff v_\tau^n\in L^{6/5}(\Rdrei)$ for every $n\in\N$,
 and the following estimate holds:
 \begin{align}
   \label{eq:time1}
   \|\dff v_\tau^n\|_{L^{6/5}(\Rdrei)} \le a\|v_0\|_{L^{6/5}(\Rdrei)}e^{-[\kappa]_\tau n\tau}(n\tau)^{-1/2} + \eps M_1,
 \end{align}
 with the system constants
 \begin{align}
   \label{eq:M1}
   a:=(1+\kappa)Y_1, \quad \text{and} \quad
   M_1:=|\phi'(0)|Y_{6/5}(1+\kappa)^{3/4}\int_0^\infty(1+\kappa)^{-s} s^{-3/4}\dd s.
 \end{align}
\end{prop}
\begin{proof}
 From the representation formula \eqref{eq:vrep} it follows that
 \begin{align*}
   \|\dff v_\tau^n\|_{L^{6/5}(\Rdrei)} &\le (1+\kappa\tau)^{-n}\|\dff\yucca_\sigma^n\|_{L^1(\Rdrei)}\|v_0\|_{L^{6/5}(\Rdrei)}\\
   &+\tau \sum_{m=1}^n (1+\kappa\tau)^{-m}\|\dff\yucca_\sigma^m\|_{L^{6/5}(\Rdrei)}\|f_\tau^{n+1-m}\|_{L^1(\Rdrei)}.
 \end{align*}
 Now apply estimate \eqref{eq:yuccaest}, once with $q:=1$ and $Q:=1/2$ to the first term,
 and once with $q:=6/5$ and $Q:=3/4$ to the second term on the right-hand side.
 Further, since $u_\tau^n$ is of unit mass, one has
 \begin{align*}
   \|f_\tau^k\|_{L^1(\Rdrei)} = \eps\|u_\tau^k\phi'(v_\tau^k)\|_{L^1(\Rdrei)} \le \eps|\phi'(0)|.
 \end{align*}
 This yields
 \begin{align}
   \label{eq:semigroup}
   \|\dff v_\tau^n\|_{L^{6/5}(\Rdrei)} \le Y_1\|v_0\|_{L^{6/5}(\Rdrei)}(1+\kappa\tau)^{-n}(\sigma n)^{-1/2}
   + \eps|\phi'(0)|Y_{6/5}\,\tau\sum_{m=1}^n (1+\kappa\tau)^{-m}(\sigma m)^{-3/4}.
 \end{align}
 The sum in \eqref{eq:semigroup} is bounded uniformly in $n$ and $\tau$ because
 \begin{align*}
   \tau\sum_{m=1}^\infty (1+\kappa\tau)^{-m}(\sigma m)^{-3/4}
   &\le (1+\kappa\tau)^{3/4}\int_0^\infty e^{-[\kappa]_\tau t}t^{-3/4}\dd t
 \end{align*}
Without loss of generality, we assume that $\tau\le 1$. By the monotone convergence $e^{-[\kappa]_\tau t}\downarrow e^{-\kappa t}$ as $\tau\downarrow 0$, we can estimate the sum in \eqref{eq:semigroup} as
\begin{align*}
   \tau\sum_{m=1}^\infty (1+\kappa\tau)^{-m}(\sigma m)^{-3/4}
   &\le (1+\tau)^{3/4}\int_0^\infty(1+\kappa)^{-t} t^{-3/4}\dd t,
 \end{align*}
and the r.h.s. is finite.
 Thus \eqref{eq:semigroup} implies \eqref{eq:time1}, with the given constants.
\end{proof}
In view of \eqref{eq:tauconverge}, we can draw the following conclusion from \eqref{eq:time1},
with $\eps_1:=\min(\eps_0,1/2)$, where $\eps_0>0$ was implicitly characterized in Proposition \ref{prop:prop_L}.
\begin{coro}
 \label{cor:T1}
 Assume that $v_0\in L^{6/5}(\Rdrei)$, and define
 \begin{align}
   \label{eq:T1}
   T_1 := \max\Big( 0,\frac1\kappa \log \frac{a\|v_0\|_{L^{6/5}(\Rdrei)}}{M_1} \Big),
 \end{align}
 with the system constants $a$ and $M_1$ from \eqref{eq:M1}.
 Then for every $\eps\in(0,\eps_1)$, for every sufficiently small $\tau$, and for every $n$ such that $n\tau\ge T_1$,
one has
 \begin{align}
   \label{eq:control1}
   \|\dff v_\tau^n\|_{L^{6/5}(\Rdrei)} \le 2 M_1.
 \end{align}
\end{coro}

\subsection{Bounds on the auxiliary entropy}
We are now in position to prove the main estimate leading towards exponential decay and boundedness of the auxiliary entropy $\lyp$ along the discrete solution.
\begin{lemma}
 \label{lem:large}
 There are system constants $L'$, $M'$ and an $\eps_2\in(0,\eps_1)$ such that
 for every $\eps\in(0,\eps_2)$, for every sufficiently small $\tau>0$,
 and for every $n$ with $n\tau>T_1$,
 we have that
 \begin{align}
   \label{eq:estimate2}
   (1+2\Lambda_\eps'\tau)\lyp(u_\tau^n,v_\tau^n) \le \lyp(u_\tau^{n-1},v_\tau^{n-1}) + \tau\eps M'
 \end{align}
 with $\Lambda_\eps':=\min(\kappa,\lambda_0)-L'\eps$.
\end{lemma}
\begin{proof}
 For brevity, we simply write $u$ and $v$ in place of $u_\tau^n$ and $v_\tau^n$, respectively,
 and we introduce $\bv:=v-v_\infty$.
 Since $n\tau>T_1$ by hypothesis, Corollary \ref{cor:T1} implies that
 \begin{align*}
   \|\dff \bv\|_{L^{6/5}(\Rdrei)} \le \|\dff v\|_{L^{6/5}(\Rdrei)} + \|\dff v_\infty\|_{L^{6/5}(\Rdrei)} \le
   Z:= 2M_1 + \sup_{0<\eps<\eps_1}\|\dff v_\infty\|_{L^{6/5}(\Rdrei)}<\infty.
 \end{align*}
 Now, since
 \begin{align*}
   \big|\dff\big(\phi(v)-\phi(v_\infty)\big)\big|^2
   \le 2\phi'(v)^2|\dff\bv|^2 + 2 (\phi'(v)-\phi'(v_\infty))^2|\dff v_\infty|^2
   \le \alpha|\dff\bv|^2 + \beta\bv^2,
 \end{align*}
 with the system constants
 \begin{align}
   \label{eq:alphabeta}
   \alpha := 2\phi'(0)^2, \quad \beta := 2\bphi^2\sup_{0<\eps<\eps_1}\|\dff v_\infty\|_{L^\infty(\Rdrei)}^2,
 \end{align}
 we conclude that
 \begin{align}
   \label{eq:chain}
   \begin{split}
     & \int_\Rdrei u\big|\dff\big(\phi(v)-\phi(v_\infty)\big)\big|^2\dd x
     \le \alpha\int_\Rdrei u|\dff\bv|^2\dd x + \beta\int_\Rdrei u\bv^2\dd x \\
     &\le \alpha\|u\|_{L^3(\Rdrei)}\|\dff\bv\|_{L^3(\Rdrei)}^2 + \beta\|u\|_{L^1(\Rdrei)}\|\bv\|_{L^\infty(\Rdrei)}^2 \\
     &\le \|u\|_{L^3(\Rdrei)}^4 + \alpha^{4/3}\|\dff\bv\|_{L^3(\Rdrei)}^{8/3} + \beta \|\bv\|_{L^\infty(\Rdrei)}^2 \\
     &\le \|u\|_{L^3(\Rdrei)}^4 + \alpha^{4/3}\big(S_1\|\bv\|_{W^{2,2}(\Rdrei)}^{3/4}\|\dff\bv\|_{L^{6/5}(\Rdrei)}^{1/4}\big)^{8/3} + \beta S_2\|\bv\|_{W^{2,2}(\Rdrei)}^2 \\
     &\le \theta \bigg(1+\int_\Rdrei u|\dff(u+W_\eps)|^2\dd x \bigg)
     + \frac{\alpha^{4/3}S_1^{8/3}Z^{2/3}+\beta S_2}{\min(1,2\kappa,\kappa^2)}\int_\Rdrei\big(\Delta\bv-\kappa\bv\big)^2\dd x,
     %
   \end{split}
 \end{align}
 where $\theta$ is the constant from \eqref{eq:u3}, and $S_1$, $S_2$ are Sobolev constants.
 Next, observe that
 \begin{align*}
   (u\phi'(v)-u_\infty\phi'(v_\infty))^2
   \le 2(u-u_\infty)^2\phi'(v)^2 + 2u_\infty^2(\phi'(v)-\phi'(v_\infty))^2
   \le \alpha(u-u_\infty)^2 + \beta\|u_\infty\|_{L^\infty(\Rdrei)}^2\bv^2,
\end{align*}
 with the same constants as in \eqref{eq:alphabeta}.
 Therefore, using \eqref{eq:subdiff_u}, \eqref{eq:subdiff_v} and Prop. \ref{prop:est_infty}(a),
 \begin{align*}
   \int_\Rdrei (u\phi'(v)-u_\infty\phi'(v_\infty))^2\dd x
   &\le \alpha\|u-u_\infty\|_{L^2(\Rdrei)}^2 + \beta(U_0-\eps V \phi'(0))^2\|\bv\|_{L^2(\Rdrei)}^2
   \\&\le 2\alpha\lyp_u(u) + \frac{2\beta}{\kappa}(U_0-\eps V \phi'(0))^2\lyp_v(v).
 \end{align*}
 Altogether, we have shown that there is a system constant $M'$
 such that (recall the dissipation terms $\dsp_u(u,v)$ and $\dsp_v(u,v)$ from \eqref{eq:du}\&\eqref{eq:dv})
 \begin{align}
   \label{eq:thedissipation}
   \begin{split}
     \dsp_u(u,v)+\dsp_v(u,v)
     & \ge (1-M'\eps)\int_\Rdrei u|\dff(u+W_\eps)|^2\dd x + (1-M'\eps) \int_\Rdrei\big(\Delta\bv-\kappa\bv\big)^2\dd x \\
     & \qquad - M'\eps \lyp_u(u) - M'\eps \lyp_v(v) - M'\eps
   \end{split}
 \end{align}
 for all $\eps\in(0,\eps_1)$.
 Provided that $M'\eps<1$, we can apply \eqref{eq:subdiff_u} and \eqref{eq:subdiff_v} to estimate further:
 \begin{align*}
   \dsp_u(u,v)+\dsp_v(u,v) \ge \big(2\lambda_\eps(1-M'\eps)-M'\eps\big)\lyp_u(u) + \big(2\kappa(1-M'\eps)-M'\eps\big)\lyp_v(v) - M'\eps.
 \end{align*}
 Finally, we can choose $\eps_2\in(0,\eps_1)$ so small that the coefficients of $\lyp_u$ and $\lyp_v$ above are nonnegative for every $\eps\in(0,\eps_2)$,
 and thus we arrive at the final estimate
 \begin{align*}
  \dsp_u(u,v)+\dsp_v(u,v) \ge 2(\min(\kappa,\lambda_\eps)-L'\eps) \lyp(u,v) - \eps M',
 \end{align*}
 with a suitable choice of $L'$.
 Now estimate \eqref{eq:1} implies \eqref{eq:estimate2} with $\Lambda_\eps'$ given as above.
\end{proof}
Diminishing $\eps_2$ such that the constant $1+K\eps_2$ in \eqref{eq:lypbyH} is less or equal to two, we derive the following explicit estimate:
\begin{prop}
 Assume that $v_0\in L^{6/5}(\Rdrei)$, and let $T_1$ be defined as in \eqref{eq:T1}.
 Then, for every $\eps\in(0,\eps_2)$, for every sufficiently small $\tau$, and for every $n$ with $n\tau>T_1$,
 the following estimate holds:
 \begin{align}
   \label{eq:time2}
   \lyp(u_\tau^n,v_\tau^n) \le 2\big(\calH(u_0,v_0)-\calH_\infty\big)e^{-2[\Lambda_\eps']_\tau(n\tau-T_1)} + \eps M_2,
 \end{align}
 with the system constant
\begin{align*}
M_2:=\frac{M'}{2\inf_{0<\eps<\eps_2}\Lambda_\eps'}.
\end{align*}
\end{prop}
\begin{proof}
 We prove a slightly refined estimate: Given $\bar n\in\N$ with $\bar n\tau\ge T_1$,
 we conclude by induction on $n\ge\bar n$ that
 \begin{align}
   \label{eq:primetime2}
   \lyp(u_\tau^n,v_\tau^n) \le 2\big(\calH(u_0,v_0)-\calH_\infty\big) (1+2\Lambda_\eps'\tau)^{-(n-\bar n)}
   + \frac{M'\eps}{2\Lambda_\eps'}\big(1-(1+2\Lambda_\eps'\tau)^{-(n-\bar n)}\big)
\end{align}
 which clearly implies \eqref{eq:time2}.
 For $n=\bar n$, \eqref{eq:primetime2} is a consequence of \eqref{eq:lypbyH}
 and the energy estimate \eqref{eq:apriori1}.
 Now assume \eqref{eq:primetime2} for some $n\ge\bar n$, and apply the iterative estimate \eqref{eq:estimate2}:
 \begin{align*}
   \lyp(u_\tau^{n+1},v_\tau^{n+1})
   &\le (1-2\Lambda_\eps'\tau)^{-1}\lyp(u_\tau^n,v_\tau^n) + (1+2\Lambda_\eps'\tau)^{-1}\tau M'\eps \\
   &\le 2\big(\calH(u_0,v_0)-\calH_\infty\big) (1+2\Lambda_\eps'\tau)^{-((n+1)-\bar n)} \\
   &\quad
   + \frac{M'\eps}{2\Lambda_\eps'}\big((1+2\Lambda_\eps'\tau)^{-1}-(1+2\Lambda_\eps'\tau)^{-((n+1)-\bar n)}\big)
   + (1+2\Lambda_\eps'\tau)^{-1}\tau M'\eps.
 \end{align*}
 Elementary calculations show that the last expression above equals to the right-hand side of \eqref{eq:primetime2},
 with $n+1$ in place of $n$.
\end{proof}
Invoking again \eqref{eq:tauconverge}, we obtain the following in analogy to Corollary \ref{cor:T1}:
\begin{coro}
 \label{cor:T2}
 Assume that $v_0\in L^{6/5}(\Rdrei)$, and define
 \begin{align}
   \label{eq:T2}
   T_2 := T_1 + \max\Big( 0,\frac1{2\Lambda_\eps'}\log\frac{2\big(\calH(u_0,v_0)-\calH_\infty\big)}{M_2}\Big).
 \end{align}
 Then for every $\eps\in(0,\eps_2)$,
 for every sufficiently small $\tau$,
 and for every $n$ such that $n\tau\ge T_2$, one has
 \begin{align}
   \label{eq:control2}
   \lyp(u_\tau^n,v_\tau^n) \le 2M_2.
 \end{align}
\end{coro}

We have thus proved that, for $t\ge T_2$, the auxiliary entropy $\lyp$ is bounded by a system constant. Next, we prove that $\lyp$ is not only bounded, but actually convergent to zero exponentially fast.
\subsection{Exponential decay for large times}
\begin{lemma}
 There is a constant $L''>L'$ and some $\eps_3\in(0,\eps_2)$ such that for every $\eps\in(0,\eps_3)$,
 for every sufficiently small $\tau>0$, and for every $n$ such that $n\tau>T_2$,
 we have
 \begin{align}
   \label{eq:decaysmall}
   (1+2\Lambda_\eps''\tau)\lyp(u_\tau^n,v_\tau^n) \le \lyp(u_\tau^{n-1},v_\tau^{n-1}),
 \end{align}
 with $\Lambda_\eps'':=\min(\lambda_0,\kappa) -L''\eps$.
\end{lemma}
\begin{proof}
 We proceed like in the proof of Lemma \ref{lem:large}, with the following modification.
 By Corollary \ref{cor:T2}, we know that
 \begin{align*}
   \lyp_u(u_\tau^n) \le \lyp(u_\tau^n,v_\tau^n) \le 2M_2.
 \end{align*}
 Using the first inequality in \eqref{eq:subdiff_u},
 we can estimate the $L^2$-norm of $u_\tau^n$ by a system constant:
 \begin{align*}
   \|u\|_{L^2(\Rdrei)}\le \|u_\infty\|_{L^2(\Rdrei)} + \|u-u_\infty\|_{L^2(\Rdrei)} \le Z:= \sup_{0<\eps<\eps_2}\|u_\infty\|_{L^2(\Rdrei)} + 2\sqrt{M_2}.
 \end{align*}
 This allows us to replace the chain of estimates \eqref{eq:chain} by a simpler one:
 \begin{align*}
   \int_\Rdrei u\big|\dff\big(\phi(v)-\phi(v_\infty)\big|^2\dd x
   \le \|u\|_{L^2(\Rdrei)}\big(\alpha\|\dff\bv\|_{L^4(\Rdrei)}^2 + \beta\|\bv\|_{L^4(\Rdrei)}^2\big)
 \end{align*}
 with the constants from \eqref{eq:alphabeta}.
 Using the Sobolev inequalities
 \begin{align*}
   \|\dff\bv\|_{L^4(\Rdrei)}\le S\|\bv\|_{W^{2,2}(\Rdrei)}
   \quad\text{and}\quad
   \|\bv\|_{L^4(\Rdrei)}\le S\|\bv\|_{W^{1,2}(\Rdrei)}
\end{align*}
 in combination with \eqref{eq:w22} and \eqref{eq:subdiff_v}, respectively, we arrive at
 \begin{align*}
   \int_\Rdrei u\big|\dff\big(\phi(v)-\phi(v_\infty)\big|^2\dd x
   \le \frac{\alpha ZS^2}{\min(1,2\kappa,\kappa^2)}\int_\Rdrei (\Delta\bv-\kappa\bv)^2\dd x + \frac{2\beta ZS^2}{\min(1,\kappa)}\lyp_v(v).
 \end{align*}
 This eventually leads to the dissipation estimate \eqref{eq:thedissipation} again, with a different constant $M'$,
 but \emph{without the constant term} $-\eps M'$.
 By means of \eqref{eq:1}, this implies \eqref{eq:decaysmall} for appropriate choices of $L''$ and $\eps_3$.
\end{proof}
By iteration of \eqref{eq:decaysmall}, starting from \eqref{eq:control2}, one immediately obtains
\begin{prop}
 For all sufficiently small $\tau$ and every $n$ such that $n\tau\ge T_2$,
 we have
 \begin{align}
   \label{eq:time3}
   \lyp(u_\tau^n,v_\tau^n) \le 2M_2e^{-2[\Lambda_\eps'']_\tau (n\tau-T_2)}.
 \end{align}
\end{prop}

\subsection{Passage to continuous time}
To complete the proof of Theorem \ref{thm:convergence}, we consider the limit $\tau\downarrow0$ of the estimates obtained above.
Here $\tau\downarrow0$ means that we consider a vanishing sequence $(\tau_k)_{k\in\N}$
such that the corresponding sequence of discrete solutions $(u_{\tau_k},v_{\tau_k})_{k\in\N}$ converges
in the sense of Section \ref{sec:existence} to a weak solution $(u,v)$ to \eqref{eq:pde_u}--\eqref{eq:init}.
Since the convergence of $(u_{\tau_k},v_{\tau_k})_{k\in\N}$ in $\theX$ is locally uniform on each compact time interval,
the lower semicontinuity of $\lyp$ in $\theX$ allows to conclude that
\begin{align*}
 \lyp(t):=\lyp\big(u(t),v(t)\big)\le\liminf_{\tau\downarrow0}\lyp\big( u_\tau(t), v_\tau(t)\big) \quad \text{for every $t\ge0$}.
\end{align*}
We prove that
\begin{align}
 \label{eq:lypdecay}
\lyp(t)&\le C(1+\|v_0\|_{L^{6/5}(\Rdrei)})^2(1+\calH(u_0,v_0)-\calH_\infty)^2e^{-2\Lambda_\eps''t}\quad \text{for all $t\ge 0$}.
\end{align}
From this, claim \eqref{eq:theclaim} in Theorem \ref{thm:convergence} follows with $\Lambda_\eps:=\Lambda_\eps''$.

Recalling \eqref{eq:tauconverge}, we conclude from \eqref{eq:time3} that
\begin{align}
  \label{eq:ctime3}
 \lyp(t) &\le 2M_2e^{-2\Lambda_\eps''(t-T_2)} \quad \text{for all $t\ge T_2$.}
\end{align}
Moreover, from \eqref{eq:lypbyH} and the energy estimate \eqref{eq:apriori1}, we obtain
\begin{align}
\lyp(t)&\le 2\big(\calH(u_0,v_0)-\calH_\infty\big)\quad \text{for all $t\ge 0$.}\label{eq:Lbound}
\end{align}
We distinguish:\\
\underline{Case 1:} $\calH(u_0,v_0)-\calH_\infty\le \frac12 M_2$.\\
Then, from the definition of $T_2$ in \eqref{eq:T2}, one has $T_2=T_1$, and in consequence of \eqref{eq:ctime3},
\begin{align*}
 \lyp(t) &\le 2M_2e^{-2\Lambda_\eps''(t-T_1)} \quad \text{for all $t\ge T_2=T_1$}.
\end{align*}
Since further $\lyp(t)\le M_2$ for all $t\ge 0$ by \eqref{eq:Lbound}, the first inequality extends to all times $t\ge 0$:
\begin{align}
 \lyp(t) &\le 2M_2e^{-2\Lambda_\eps''(t-T_1)} \quad \text{for all $t\ge 0$}.\label{eq:case1}
\end{align}
\underline{Case 2:} $\calH(u_0,v_0)-\calH_\infty>\frac12 M_2$.\\
Substitute
\begin{align*}
T_2&=T_1+\frac{1}{2\Lambda_\eps'}\log\left(\frac{2(\calH(u_0,v_0)-\calH_\infty)}{M_2}\right)
\end{align*}
in \eqref{eq:ctime3} to find
\begin{align}
 \lyp(t) \le 2M_2e^{-2\Lambda_\eps''(t-T_1)}\left(\frac{2(\calH(u_0,v_0)-\calH_\infty)}{M_2}\right)^{\Lambda_\eps''/\Lambda_\eps'} \quad \text{for all $t\ge T_2>T_1$}.\label{eq:case2prae}
\end{align}

Using $\Lambda_\eps''\le\Lambda_\eps'$, we conclude from \eqref{eq:case2prae}:
\begin{align}
 \lyp(t) &\le 4(\calH(u_0,v_0)-\calH_\infty)e^{-2\Lambda_\eps''(t-T_1)} \quad \text{for all $t\ge T_2>T_1$}\label{eq:case2prae2}.
\end{align}

Define $A:=4(\calH(u_0,v_0)-\calH_\infty)\max\left(1,\frac{\calH(u_0,v_0)-\calH_\infty}{M_2}\right)$. Then, from \eqref{eq:case2prae2} and the fact that $Ae^{-2\Lambda_\eps''(T_2-T_1)}\ge 2(\calH(u_0,v_0)-\calH_\infty)$, we deduce
\begin{align}
\lyp(t) &\le Ae^{-2\Lambda_\eps''(t-T_1)} \quad \text{for all $t\ge 0$}.\label{eq:case3}
\end{align}

Together, \eqref{eq:case1} and \eqref{eq:case3} yield
\begin{align*}
 \lyp(t) &\le \max(2M_2,4(\calH(u_0,v_0)-\calH_\infty))\max\left(1,\frac1{M_2}(\calH(u_0,v_0)-\calH_\infty)\right)e^{-2\Lambda_\eps''(t-T_1)}\\&\le\frac{2}{M_2}e^{2\Lambda_\eps''T_1}\max(M_2,2(\calH(u_0,v_0)-\calH_\infty))^2e^{-2\Lambda_\eps''t}
\quad \text{for all $t\ge 0$}.
\end{align*}
Since $\kappa\ge\Lambda_\eps''$, we have
\begin{align*}
e^{2\Lambda_\eps''T_1}&\le e^{2\kappa T_1}\le \max\left(1,\left[\frac{a}{M_1}\|v_0\|_{L^{6/5}(\Rdrei)}\right]^2\right),
\end{align*}
and consequently \eqref{eq:lypdecay}:
\begin{align*}
\lyp(t)&\le C(1+\|v_0\|_{L^{6/5}(\Rdrei)})^2(1+\calH(u_0,v_0)-\calH_\infty)^2e^{-2\Lambda_\eps''t}\quad \text{for all $t\ge 0$}.
\end{align*}
\qed

\appendix
\section{Proof of Lemma \ref{lemma:reg_yuk}}
\label{app:yukawa}
\begin{enumerate}[(a)]
\item The proof of the first assertion can be found in \cite[Thm. 6.23]{lieb2001}. From that, the second one follows by elementary calculations.
\item According to \cite[Ch. V, §3.3, Thm. 3]{stein1971}, one has for $p>1$:
\begin{align}
\|\krnl_1*f\|_{W^{2,p}(\Rdrei)}&\le C_p \|f\|_{L^p(\Rdrei)}.\label{eq:stein}
\end{align}
To prove assertion (b), we use a rescaling of the equation $-\Delta h+\kappa h=f$ by $\tilde x:=\sqrt{\kappa}x$. Consequently, $h(\tilde x)=(\krnl_\kappa*f)\left(\frac{\tilde x}{\sqrt{\kappa}}\right)$ is a solution to $-\Delta_{\tilde x}h+h=\frac{f}{\kappa}$, i.e. $h(\tilde x)=\left(\krnl_1*\frac{f}{\kappa}\right)(\tilde x)$. By the transformation theorem, we obtain
\begin{align}
\left(\int_\Rdrei\left|\frac{f(\tilde x)}{\kappa}\right|^p\dd\tilde x\right)^\frac{1}{p}&=\kappa^{\frac{3}{2}-1}\|f\|_{L^p(\Rdrei)},\nonumber\\
\left(\int_\Rdrei\left|\left(\krnl_1*\frac{f}{\kappa}\right)(\tilde x)\right|^p\dd\tilde x\right)^\frac{1}{p}&=\kappa^{\frac{3}{2}}\|\krnl_\kappa*f\|_{L^p(\Rdrei)},\nonumber\\
\left(\int_\Rdrei\left|\dff _{\tilde x}\left(\krnl_1*\frac{f}{\kappa}\right)(\tilde x)\right|^p\dd\tilde x\right)^\frac{1}{p}&=\kappa^{\frac{3}{2}-\frac{1}{2}}\|\dff _x(\krnl_\kappa*f)\|_{L^p(\Rdrei)},\nonumber\\
\left(\int_\Rdrei\left|\dff ^2_{\tilde x}\left(\krnl_1*\frac{f}{\kappa}\right)(\tilde x)\right|^p\dd\tilde x\right)^\frac{1}{p}&=\kappa^{\frac{3}{2}-1}\|\dff ^2_x(\krnl_\kappa*f)\|_{L^p(\Rdrei)},\nonumber
\end{align}
which yields \eqref{eq:reg_yuk} after insertion in \eqref{eq:stein} and simplification.
\item The first statement is a straight-forward consequence of the integral-type representation of $\krnl_\kappa$ in \cite[Thm. 6.23]{lieb2001}. To prove the first claim of the second statement, we proceed by induction.
 For $k=1$, equation \eqref{eq:yuccaiter} is just the definition of $\yucca_\sigma$.
 Now assume that \eqref{eq:yuccaiter} holds for some $k\in\N$.
 Using the semigroup property $\heat_{t_1+t_2}=\heat_{t_1} *\heat_{t_2}$ of the heat kernel,
 we find that
 \begin{align*}
   \yucca_\sigma^{k+1}
   &= \int_0^\infty \int_0^\infty \heat_{\sigma r_1} *\heat_{\sigma r_2}e^{-r_1}r_2^{k-1}e^{-r_2}\frac{\dd r_1\dd r_2}{\Gamma(k)} \\
   &= \int_0^\infty\int_0^\infty \heat_{\sigma(r_1+r_2)}e^{-(r_1+r_2)} r_2^{k-1}\frac{\dd r_1\dd r_2}{\Gamma(k)}.
 \end{align*}
 Now perform a change of variables
 \begin{align*}
   r:=r_1+r_2,\ s:=r_2,
 \end{align*}
 which is of determinant $1$ and leads to
 \begin{align*}
   \yucca_\sigma^{k+1}
   = \int_0^\infty \heat_{\sigma r}e^{-r}\bigg(\int_0^r s^{k-1}\dd s\bigg)\frac{\dd r}{\Gamma(k)}
   = \int_0^\infty \heat_{\sigma r}\frac{e^{-r}r^k\dd r}{k\Gamma(k)},
 \end{align*}
 which is \eqref{eq:yuccaiter} with $k+1$ in place of $k$, using that $k\Gamma(k)=\Gamma(k+1)$.

 For \eqref{eq:yuccaest}, first observe that $r\mapsto r^{k-1}e^{-r}/\Gamma(k)$ defines a probability density on $\R_+$.
 We can thus apply Jensen's inequality to obtain
 \begin{align}
   \label{eq:ghj}
   \|\dff\yucca_\sigma^k\|_{L^q(\Rdrei)}
   \le \int_0^\infty \|\dff\heat_{\sigma r}\|_{L^q(\Rdrei)}\frac{r^{k-1}e^{-r}\dd r}{\Gamma(k)} .
\end{align}
 The $L^q$-norm of $\dff\heat_{\sigma r}$ is easily evaluated using its definition,
 \begin{align*}
   \|\dff\heat_{\sigma r}\|_{L^q(\Rdrei)}
   &= (\sigma r)^{-3/2}\bigg(\int_\Rdrei |\dff_\xi\heat_1\big((\sigma r)^{-1/2}\xi\big)|^q\dd\xi\bigg)^{1/q} \\
   & = (\sigma r)^{-3/2}\bigg(\int_\Rdrei |(\sigma r)^{-1/2}\dff_\zeta\heat_1(\zeta)|^q\,(\sigma r)^{3/2}\dd\zeta\bigg)^{1/q}
   = (\sigma r)^{-Q}\|\dff\heat_1\|_{L^q(\Rdrei)}.
 \end{align*}
 By definition of the gamma function, we thus obtain from \eqref{eq:ghj} that
 \begin{align*}
   \|\dff\yucca_\sigma^k\|_{L^q(\Rdrei)} \le \|\dff\heat_1\|_{L^q(\Rdrei)} \frac{\Gamma(k-Q)}{\Gamma(k)}\sigma^{-Q}.
 \end{align*}
 For further estimation, observe that the sequence $(a_k)_{k\in\N}$ with $a_k=k^Q\Gamma(k-Q)/\Gamma(k)$ is monotonically decreasing (to zero).
 Indeed,
 \begin{align*}
   \frac{a_{k+1}}{a_k} = \frac{(k+1)k^Q\,(k-Q)\Gamma(k-Q)\Gamma(k)}{k^Q\,k\Gamma(k)\Gamma(k-Q)} = \Big(1+\frac1k\Big)^Q\Big(1-\frac{Q}k\Big)
 \end{align*}
 is always smaller than $1$ since $\xi\mapsto(1+\xi)^{-Q}$ is convex.
 Therefore $a_k\le a_1$ for all $k\in\N$, and so \eqref{eq:yuccaest} follows with $Y_q:=\Gamma(1-Q)\|\dff\heat_1\|_{L^q(\Rdrei)}$.\qed
\end{enumerate}

\section{H\"older estimate for the kernel $\krnl_\kappa$}\label{app:holder}

As a preparation, we calculate the derivatives of $\krnl_\kappa$ in $\Rdrei\backslash\{0\}$. For all $i,j,k\in\{1,2,3\}$, one has
\begin{align*}
\partial_i \krnl_\kappa(x)&=-\frac{1}{4\pi}\frac{\exp(-\sqrt{\kappa} |x|)}{|x|^3}(\sqrt{\kappa} |x|+1)x_i,\\
\partial_i\partial_j \krnl_\kappa(x)&=-\frac{1}{4\pi}\exp(-\sqrt{\kappa} |x|)\left[\left(\frac{\kappa}{|x|^3}+\frac{3\sqrt{\kappa}}{|x|^4}+\frac{3}{|x|^5}\right)x_i x_j-\left(\frac{\sqrt{\kappa}}{|x|^2}+\frac{1}{|x|^3}\right)\delta_{ij}\right],\\
\partial_i\partial_j\partial_k \krnl_\kappa(x)&=-\frac{1}{4\pi}\exp(-\sqrt{\kappa} |x|)\frac{-\sqrt{\kappa} x_k}{|x|}\left[\left(\frac{\kappa}{|x|^3}+\frac{3\sqrt{\kappa}}{|x|^4}+\frac{3}{|x|^5}\right)x_i x_j-\left(\frac{\sqrt{\kappa}}{|x|^2}+\frac{1}{|x|^3}\right)\delta_{ij}\right]\nonumber\\
&-\frac{1}{4\pi}\exp(-\sqrt{\kappa} |x|)\left[\left(-\frac{3\kappa}{|x|^4}-\frac{12\sqrt{\kappa}}{|x|^5}-\frac{15}{|x|^6}\right)\frac{x_i x_j x_k}{|x|}\right.\nonumber\\
&\left.+\delta_{ij}\left(\frac{2\sqrt{\kappa}}{|x|^3}+\frac{3}{|x|^4}\right)\frac{x_k}{|x|}+\left(\frac{\kappa}{|x|^3}+\frac{3\sqrt{\kappa}}{|x|^4}+\frac{3}{|x|^5}\right)(\delta_{ik}x_j+\delta_{jk}x_i)\right],
\end{align*}
where $\delta_{ij}$ denotes Kronecker's delta.

We prove the following
\begin{lemma}[H\"older estimate for second derivative]\label{lemma:holder}
Let $f\in C^{0,\alpha}(\Rdrei)$ for some $\alpha\in(0,1)$ and assume that it is of compact support. Then, there exists $C>0$ such that for all $i,j\in\{1,2,3\}$ the following estimate holds:
\begin{align*}
\left[\partial_i\partial_j (\krnl_\kappa*f)\right]_{C^{0,\alpha}(\Rdrei)}&\le C[f]_{C^{0,\alpha}(\Rdrei)}.
\end{align*}
Here,
\begin{align*}
[g]_{C^{0,\alpha}(\Rdrei)}:=\sup_{x,y\in\R^3,\, x\neq y}\frac{|g(x)-g(y)|}{|x-y|}
\end{align*}
denotes the \emph{H\"older seminorm} of $g:\,\R^3\to\R$.
\end{lemma}

\begin{proof}
This result is an extension of the respective result for Poisson's equation (corresponding to $\kappa=0$) proved by Lieb and Loss \cite[Thm. 10.3]{lieb2001}. Their method of proof is adapted here. In the following, $C,\tilde C$ denote generic nonnegative constants.

The following holds for arbitrary test functions $\psi\in C^\infty_c(\Rdrei)$:
\begin{align*}
-\int_\Rdrei (\partial_j \psi)(x)(\partial_i v)(x)\dd x=\int_\Rdrei f(y)\int_\Rdrei (\partial_j \psi)(x)\partial_{x_i}\krnl_\kappa(x-y)\dd x\,\dd y,
\end{align*}
which can be rewritten by the dominated convergence theorem and integration by parts as
\begin{align*}
&\int_\Rdrei f(y)\int_\Rdrei (\partial_j \psi)(x)\partial_{x_i}\krnl_\kappa(x-y)\dd x\,\dd y\\&=\lim_{\delta\to 0}\int_\Rdrei f(y)\int_{\Rdrei\backslash B_\delta(y)} (\partial_j \psi)(x)\partial_{x_i}\krnl_\kappa(x-y)\dd x\,\dd y\nonumber\\
&=\lim_{\delta\to 0}\int_\Rdrei f(y)\left[-\int_{\partial B_\delta(y)}\psi(x)\partial_{x_i}\krnl_\kappa(x-y)e_j\cdot \nu_{y,\delta}(x)\dd S(x)\right.\nonumber\\&\left.-\int_{\Rdrei\backslash B_\delta(y)} \psi(x)\partial_{x_i}\partial_{x_j}\krnl_\kappa(x-y)\dd x\right]\dd y,
\end{align*}
where $e_j$ is the $j$-th unit vector and $\nu_{y,\delta}(x)=\frac{x-y}{\delta}$ is the unit outward normal vector in $x$ on the sphere $\partial B_\delta(y)$.
The first part can be simplified explicitly by the transformation $z:=\frac{x-y}{\delta}$:
\begin{align*}
&-\int_{\partial B_\delta(y)}\psi(x)\partial_{x_i}\krnl_\kappa(x-y)e_j\cdot \nu_{y,\delta}(x)\dd S(x)\\&=\frac{1}{4\pi}\int_{\partial B_1(0)}\psi(\delta z+y)\exp(-\sqrt{\kappa}\delta)(\sqrt{\kappa}\delta+1)z_i z_j \dd S(z),
\end{align*}
which converges as $\delta\to 0$ to $\psi(y)\frac{\delta_{ij}}{3}$.
For the second part, we split the domain of integration $\Rdrei\backslash B_\delta(y)$ in two parts:
\begin{align*}
&\int_{\Rdrei\backslash B_\delta(y)}\psi(x)\partial_{x_i}\partial_{x_j}\krnl_\kappa(x-y)\dd x\\&=\int_{\Rdrei\backslash B_\delta(1)}\psi(x)\partial_{x_i}\partial_{x_j}\krnl_\kappa(x-y)\dd x+\int_{\{1\ge|x-y|\ge \delta\}}\psi(x)\partial_{x_i}\partial_{x_j}\krnl_\kappa(x-y)\dd x.
\end{align*}
We use integration by parts to insert convenient additional terms:
\begin{align*}
&\int_{\Rdrei\backslash B_\delta(1)}\psi(x)\partial_{x_i}\partial_{x_j}\krnl_\kappa(x-y)\dd x+\int_{\{1\ge|x-y|\ge \delta\}}\psi(x)\partial_{x_i}\partial_{x_j}\krnl_\kappa(x-y)\dd x=\nonumber\\
&=\int_{\Rdrei\backslash B_\delta(1)}\psi(x)\partial_{x_i}\partial_{x_j}\krnl_\kappa(x-y)\dd x+\int_{\{1\ge|x-y|\ge \delta\}}\psi(x)\partial_{x_i}\partial_{x_j}\krnl_\kappa(x-y)\dd x\nonumber\\
&-\int_{\{1\ge|x-y|\ge \delta\}}\psi(y)\partial_{x_i}\partial_{x_j}\krnl_\kappa(x-y)\dd x\nonumber\\
&+\int_{\partial B_1(y)}\psi(y)\partial_{x_j} \krnl_\kappa(x-y)e_i\cdot \nu_{y,1}(x)\dd S(x)-\int_{\partial B_\delta(y)}\psi(y)\partial_{x_j} \krnl_\kappa(x-y)e_i\cdot \nu_{y,\delta}(x)\dd S(x).
\end{align*}
Now we calculate again explicitly and obtain in the limit $\delta\to 0$:
\begin{align*}
&\int_{\partial B_1(y)}\psi(y)\partial_{x_j} \krnl_\kappa(x-y)e_i\cdot \nu_{y,1}(x)\dd S(x)-\int_{\partial B_\delta(y)}\psi(y)\partial_{x_j} \krnl_\kappa(x-y)e_i\cdot \nu_{y,\delta}(x)\dd S(x)\nonumber\\
&\longrightarrow -\frac{1}{3}\delta_{ij}[\exp(-\sqrt{\kappa})(\sqrt{\kappa}+1)-1]\psi(y).
\end{align*}

In summary, one gets
\begin{align*}
&-\int_\Rdrei (\partial_j \psi)(x)(\partial_i v)(x)\dd x\nonumber\\
&=\int_\Rdrei\psi(x)\left[\frac{1}{3}\delta_{ij}f(x)\exp(-\sqrt{\kappa})(\sqrt{\kappa}+1)+\int_{\Rdrei\backslash B_1(x)}f(y)\partial_{x_i}\partial_{x_j}\krnl_\kappa(x-y)\dd y\right.\nonumber\\
&\left.+\lim_{\delta\to 0}\int_{\{1\ge|x-y|\ge \delta\}} (f(x)-f(y))\partial_{x_i}\partial_{x_j}\krnl_\kappa(x-y)\dd y\right]\dd x.
\end{align*}

From $\alpha$-H\"older continuity of $f$, we conclude that, independent of $\delta$,
\begin{align*}
 \mathbf{1}_{{\{1\ge|x-y|\ge \delta\}}}(y)\left|[f(x)-f(y)]\partial_{x_i}\partial_{x_j}\krnl_\kappa(x-y)\right|\le C|x-y|^{\alpha-3},
\end{align*}
which is integrable as $\alpha-3+2>-1$. So, using again the dominated convergence theorem, we have, with \cite[Thm. 6.10]{lieb2001},
\begin{align}
(\partial_i\partial_j v)(x)&=\frac{1}{3}\delta_{ij}\exp(-\sqrt{\kappa})(\sqrt{\kappa}+1)\nonumber\\&+\int_{\Rdrei\backslash B_1(x)}f(y)\partial_{x_i}\partial_{x_j}\krnl_\kappa(x-y)\dd y+\int_{B_1(x)}[f(x)-f(y)]\partial_{x_i}\partial_{x_j}\krnl_\kappa(x-y)\dd y.\label{eq:D2v}
\end{align}
Obviously, the first term in \eqref{eq:D2v} is H\"older-continuous. For the second term in \eqref{eq:D2v}, we obtain for all $x,z\in \Rdrei$, $x\neq z$:
\begin{align*}
&\left|\int_{\Rdrei\backslash B_1(x)}f(y)\partial_{x_i}\partial_{x_j}\krnl_\kappa(x-y)\dd y-\int_{\Rdrei\backslash B_1(z)}f(y)\partial_{z_i}\partial_{z_j}\krnl_\kappa(z-y)\dd y \right|\nonumber\\
&=\left|\int_{B_1(0)}[f(z-a)-f(x-a)]\partial_{a_i}\partial_{a_j}\krnl_\kappa(a)\dd a\right|,
\end{align*}
by the transformation $a:=x-y$ in the first and $a:=z-y$ in the second integral. From $\alpha$-H\"older continuity of $f$, we get the estimate
\begin{align*}
\left|\int_{B_1(0)}[f(z-a)-f(x-a)]\partial_{a_i}\partial_{a_j}\krnl_\kappa(a)\dd a\right|&\le C|x-z|^\alpha \int_{\Rdrei\backslash B_1(x)}|\partial_{a_i}\partial_{a_j}\krnl_\kappa(a)\dd a|,
\end{align*}
where the integral on the r.h.s. is finite because $\partial_{a_i}\partial_{a_j}\krnl_\kappa(a)$ behaves as $r^{-1}\exp(-r)$ for $r\to\infty$, which is integrable.

The same integral transformation yields for the third term in \eqref{eq:D2v}:
\begin{align*}
&\left|\int_{B_1(x)}[f(x)-f(y)]\partial_{x_i}\partial_{x_j}\krnl_\kappa(x-y)\dd y-\int_{B_1(z)}[f(z)-f(y)]\partial_{z_i}\partial_{z_j}\krnl_\kappa(z-y)\dd y\right|\nonumber\\
&=\left|\int_{B_1(0)}[f(z)-f(z-a)-f(x)+f(x-a)]\partial_{a_i}\partial_{a_j}\krnl_\kappa(a)\dd a\right|.
\end{align*}
We now proceed as in \cite{lieb2001} and write $B_1(0)=A\cup B$ with
\begin{align*}
A&:=\{a:\,0\le|a|<4|x-z|\},\\
B&:=\{a:\,4|x-z|<|a|<1\},
\end{align*}
where $B=\emptyset$ for $|x-z|\ge \frac{1}{4}$, and calculate, using that $|\partial_{a_i}\partial_{a_j}\krnl_\kappa(a)|\le C|a|^{-3}$:
\begin{align*}
&\left|\int_{A}[f(z)-f(z-a)-f(x)+f(x-a)]\partial_{a_i}\partial_{a_j}\krnl_\kappa(a)\dd a\right|\le \int_A 2C|a|^{\alpha-3}\dd a=\tilde C|x-z|^\alpha.
\end{align*}

It remains to consider the case $|x-z|< \frac{1}{4}$. One has
\begin{align*}
&\left|\int_B[f(z)-f(x)]\partial_{a_i}\partial_{a_j}\krnl_\kappa(a) \dd a\right|=\left|\int_{\partial B}[f(z)-f(x)]\partial_{a_j}\krnl_\kappa(a)e_i\cdot \nu(a)\dd S(a)\right|,
\end{align*}
and by similar arguments as above,
\begin{align*}
&\left|\int_{\partial B}[f(z)-f(x)]\partial_{a_j}\krnl_\kappa(a)e_i\cdot \nu(a)\dd S(a)\right|\nonumber\\
&=\frac{1}{3}\delta_{ij}|f(z)-f(x)||\exp(-\sqrt{\kappa})(\sqrt{\kappa}+1)-\exp(-4\sqrt{\kappa}|x-z|)(4\sqrt{\kappa}|x-z|+1)|.
\end{align*}
Note that the real-valued map $[0,\infty)\ni r\mapsto \exp(-\sqrt{\kappa} r)(\sqrt{\kappa} r+1)$ is monotonically decreasing. This yields
\begin{align*}
&\left|\int_B[f(z)-f(x)]\partial_{a_i}\partial_{a_j}\krnl_\kappa(a) \dd a\right|\le C|z-x|^\alpha.
\end{align*}

By the transformations $b:=x-a-z$ and $b:=-a$, we get
\begin{align}
&\left|\int_B [f(x-a)-f(z-a)]\partial_{a_i}\partial_{a_j}\krnl_\kappa(a)\dd a\right|\nonumber\\&=\left|\int_B f(z+b)\partial_{b_i}\partial_{b_j}\krnl_\kappa(b)\dd b-\int_D f(b+z)\partial_{b_i}\partial_{b_j}\krnl_\kappa(b-x+z)\dd b\right|,\label{eq:diff}
\end{align}
with $D:=\{b:\,4|x-z|<|b-x+z|<1\}$.

Note that
\begin{align*}
\int_B\partial_{b_i}\partial_{b_j}\krnl_\kappa(b)\dd b&=\int_D\partial_{b_i}\partial_{b_j}\krnl_\kappa(b-x+z)\dd b.
\end{align*}
This enables us to rewrite \eqref{eq:diff} as follows:
\begin{align}
&\left|\int_B f(z+b)\partial_{b_i}\partial_{b_j}\krnl_\kappa(b)\dd b-\int_D f(b+z)\partial_{b_i}\partial_{b_j}\krnl_\kappa(b-x+z)\dd b\right|\nonumber\\
&\left|\int_B [f(z+b)-f(z)]\partial_{b_i}\partial_{b_j}\krnl_\kappa(b)\dd b-\int_D [f(z+b)-f(z)]\partial_{b_i}\partial_{b_j}\krnl_\kappa(b-x+z)\dd b\right|.\label{eq:diff2}
\end{align}

We consider \eqref{eq:diff2} separately on the sets $B\cap D$, $B\backslash D$ and $D\backslash B$.

Note that, by the triangular inequality,  $B\cap D\subset \{b:\,3|x-z|<|b|<1+|x-z|\}$ and by Taylor's theorem
\begin{align*}
(\partial_{b_i}\partial_{b_j}\krnl_\kappa)(b)-(\partial_{b_i}\partial_{b_j}\krnl_\kappa)(b-x+z)&=\sum_{k=1}^3 (\partial_k\partial_i\partial_j \krnl_\kappa)(b^*)(x_k-z_k),
\end{align*}
for some $b^*=b-\beta (x-z)$ with $\beta\in(0,1)$. Therefore, one has by the triangular inequality, $|b^*|\ge |b|-\beta|x-z|\ge (1-\frac{\beta}{3})|b|\ge \frac{2}{3}|b|$ on $B\cap D$ and consequently
\begin{align*}
|(\partial_{b_i}\partial_{b_j}\krnl_\kappa)(b)-(\partial_{b_i}\partial_{b_j}\krnl_\kappa)(b-x+z)|&\le C|b^*|^{-4}|x-z|\le \tilde C |b|^{-4}|x-z|.
\end{align*}
This entails us to estimate
\begin{align*}
&\left|\int_{B\cap D} \left[f(z+b)-f(z)\right]\left[\partial_{b_i}\partial_{b_j}\krnl_\kappa(b)-\partial_{b_i}\partial_{b_j}\krnl_\kappa(b-x+z)\right]\dd b\right|\\&\le C|x-z|\int_{3|x-z|}^{1+|x-z|}r^{-4+\alpha+2}\dd r\nonumber\\
&\le \frac{C|x-z|}{1-\alpha}\left[(3|x-z|)^{\alpha-1}-(1+|x-z|)^{\alpha-1}\right]\frac{\tilde C}{1-\alpha}|x-z|^\alpha.
\end{align*}
For the remaining terms, we split up as in \cite{lieb2001}:
\begin{align*}
B\backslash D&\subset E\cup G,\\
D\backslash B&\subset E'\cup G',
\end{align*}
where
\begin{align*}
E&:=\{b:\,4|x-z|<|b|\le 5|x-z|\},\\
G&:=\{b:\,1-|x-z|\le|b|<1\},\\
E'&:=\{b:\,4|x-z|<|b-x+z|\le 5|x-z|\},\\
G'&:=\{b:\,1-|x-z|\le |b-x+z|<1\}.
\end{align*}

Consider at first the real-valued map $[0,\frac{1}{4}]\ni s\mapsto (1-s)^\beta$ for arbitrary $\beta>0$. Obviously, it is continuously differentiable and therefore $\alpha$-H\"older continuous because its domain of definition is compact. Hence, the following holds for all $0\le s\le \frac{1}{4}$:
\begin{align}
1-(1-s)^\beta&=(1-0)^\beta-(1-s)^\beta\le C s^\alpha.\label{eq:interg}
\end{align}

Now, we estimate the integral on $B\backslash D$, where we use again the estimate $|\partial_{a_i}\partial_{a_j}\krnl_\kappa(a)|\le C|a|^{-3}$:

\begin{align*}
&\left|\int_{B\backslash D} [f(z+b)-f(z)]\partial_{b_i}\partial_{b_j}\krnl_\kappa(b)\dd b\right|\le C\left(\int_{4|x-z|}^{5|x-z|}r^{\alpha-3+2}\dd r+\int_{1-|x-z|}^{1}r^{\alpha-3+2}\dd r\right)\nonumber\\
&=\frac{C}{\alpha}\left[(5|x-z|)^\alpha-(4|x-z|)^\alpha+1-(1-|x-z|)^\alpha\right]\le \frac{C}{\alpha}(5^\alpha+\tilde C)|x-z|^\alpha,
\end{align*}
where we have used \eqref{eq:interg} for $\beta:=\alpha$ in the last step.

For the remaining integral on $D\backslash B$, we consider the domains $E'$ and $G'$ separately and note at first that, using the triangular inequality, $E'\subset \{0<|b|\le 6|x-z|\}$. Subsequently, this yields that $|b-x+z|^{-3}<(4|x-z|)^{-3}\le C|b|^{-3}$ on $E'$. Hence, by the estimate $|\partial_{a_i}\partial_{a_j}\krnl_\kappa(a)|\le C|a|^{-3}$, the following holds:
\begin{align*}
&\int_{E'} \left|[f(z+b)-f(z)]\partial_{b_i}\partial_{b_j}\krnl_\kappa(b-x+z)\right|\dd b\le C\int_0^{6|x-z|}r^{\alpha-3+2}dr=\tilde C |x-z|^\alpha.
\end{align*}
On $G'$, one has $|b-x+z|\ge 1-|x-z|>\frac{3}{4}$. Consequently, it holds that
\begin{align*}
&\int_{G'} \left|[f(z+b)-f(z)]\partial_{b_i}\partial_{b_j}\krnl_\kappa(b-x+z)\right|\dd b\le C\left(\frac{3}{4}\right)^{-3}\int_{1-|x-z|}^1r^{\alpha+2}\dd r\nonumber\\
&=\tilde C \left[1-(1-|x-z|)^{3+\alpha}\right]\le \tilde C|x-z|^\alpha,
\end{align*}
where we have used \eqref{eq:interg} for $\beta:=3+\alpha$ in the last step. Together,
\begin{align*}
\left|\int_{D\backslash B} [f(z+b)-f(z)]\partial_{b_i}\partial_{b_j}\krnl_\kappa(b-x+z)\dd b\right|&\le C|x-z|^\alpha,
\end{align*}
and the assertion is proved.
\end{proof}

\bibliographystyle{abbrv}
\bibliography{ref3.bib}
\end{document}